\documentclass[12pt]{article}
\usepackage[utf8]{inputenc}
\usepackage{amsmath}
\usepackage[T1]{fontenc}
\usepackage[english]{babel}
\usepackage{hyphenat}
\usepackage{floatrow}
\usepackage{hyperref}

\usepackage{enumerate}

\usepackage{graphicx}
\graphicspath{ {./Images/} }
\usepackage{array}
\usepackage{caption}
\usepackage{amsfonts, amsthm, xcolor}
\usepackage{url}
\usepackage[margin=1in]{geometry}
\usepackage{tikz}

\newcommand{\remove}[1]{}

\newcommand{\Co}{\mathbb{C}}

\newcommand{\N}{\mathbb{N}}

\newcommand{\Z}{\mathbb{Z}}
\newcommand{\Pp}{\mathbb{P}}
\newcommand{\un}{\mathbf{1}}

\newtheorem{Lemma}{Lemma}
\newtheorem{Statement}{Statement}
\newtheorem{Theorem}{Theorem}

\newtheorem{Claim}{Claim}
\newtheorem{Definition}{Definition}
\newtheorem{Notation}{Notation}
\newtheorem{Question}{Question}
\newtheorem{Conjecture}{Conjecture}

\newpage

\title{On the Minimum Possible Maximum Degree of Induced Subgraphs of Product Graphs}

\author{%
  Elizaveta Popova%
  \thanks{Research supported in part by ISF grant 2364/25 and Minerva grant 715025.}\\
  \small Weizmann Institute of Science, Israel
}

\date{}

\begin{document}

\maketitle
\begin{abstract}
  The following is a natural and fundamental question for a graph $G$: if an induced subgraph $H$ of $G$ has $x$ more vertices than a maximum independent set, what can be said about the maximum degree of $H$ as a function of $x$?

  The case $x=1$ is already of considerable interest. For example, in his celebrated proof of the sensitivity conjecture, Hao Huang showed that every induced subgraph of the hypercube $Q_n$ on more than $2^{n-1}$ vertices has maximum degree at least $\sqrt{n}$. Chung, F\'{u}redi, Graham, and Seymour proved that this bound is tight.

  In this paper, we study this question when $G$ is either the $n$-fold Hamming product or the $n$-fold tensor product of a triangle. For both graphs, we determine the exact minimum possible average degree of an induced subgraph of a prescribed size. We also prove that the tensor product exhibits several Huang-like phenomena. For the Hamming product, we show that, for several size densities and large $n$, the minimum possible maximum degree is asymptotically equal to the minimum possible average degree. Finally, we extend several of the results from the triangle to an arbitrary complete graph $K_k$.
\end{abstract}
\newpage
\tableofcontents
\newpage
\section{Introduction}

Huang proved the following celebrated theorem~\cite{Huang}.

\begin{Theorem}
  Every induced subgraph of the hypercube $Q_n$ on $2^{n-1}+1$ vertices has maximum degree at least $\sqrt{n}$.
\end{Theorem}

This theorem completed the proof of the long-standing sensitivity conjecture of Nisan and Szegedy in theoretical computer science~\cite{SensConj}.

Huang's theorem is also interesting as an independent combinatorial statement and naturally suggests analogous questions for other graphs. Given a graph $G$, one may ask for the minimum possible maximum degree of an induced subgraph on $\alpha(G)+1$ vertices and, in particular, whether there is a ``jump'' analogous to that in Huang's theorem. Alon and Zheng answered this question affirmatively for every bipartite Cayley graph of $\Z_2^n$ (including the hypercube)~\cite{CayleyZ2AZ}:

\begin{Theorem}\cite{CayleyZ2AZ}
  For any Cayley graph $G=\Gamma(\Z_2^n,S)$ of $\Z_2^n$ with respect to a generating set $S$, and for any subset $U\subseteq\Z_2^n$ with $|U|>2^{n-1}$, the induced subgraph $H=G[U]$ satisfies $\Delta(H)\geq\sqrt{|S|}$.
\end{Theorem}

Potechin and Tsang subsequently obtained an affirmative answer for every bipartite Cayley graph of an abelian group~\cite{AbCayleyPT}. They proved the following theorem.

\begin{Theorem}\cite{AbCayleyPT}
  Let $X=\Gamma(G,S)$ be a Cayley graph of an abelian group $G$. If $U\subseteq G$ and $|U|>|G|/2$, then the induced subgraph $X[U]$ has maximum degree at least $\sqrt{(|S|+t)/2}$, where $t$ is the number of elements of order~$2$ in $S$.
\end{Theorem}

They also conjectured that the same conclusion holds for all bipartite Cayley graphs. Lehner and Verret gave the first counterexamples~\cite{CayleyLV}; Garc\'{i}a-Marco and Knauer subsequently constructed three infinite families of counterexamples containing induced matchings on $\alpha(G)+1$ vertices~\cite{CayleyGK}.

Garc\'{i}a-Marco and Knauer also established Huang-like behavior in other Cayley graphs of non-abelian groups, including Coxeter groups. In addition, they used the expander mixing lemma to prove the existence of a jump in tensor products of an expander with an edge. We use the same lemma in a related way. In our non-bipartite setting, however, it does not yield the result directly; instead, it constrains the structure of vertex sets that induce subgraphs of small maximum degree. See Sections~\ref{bipart_jump} and~\ref{gap1}.

The question can be generalized further by considering induced subgraphs on larger vertex sets, not necessarily of size $\alpha(G)+1$. We study the following quantities.

\begin{Notation}
  For a nonempty graph $H$, let $\Delta(H)$ denote its maximum degree and let $a(H):=2|E(H)|/|V(H)|$ denote its average degree. For a graph $G$ and a set $F\subseteq V(G)$, let $G[F]$ be the subgraph of $G$ induced by $F$. For an integer $1\leq m\leq |V(G)|$, define
  $$\Delta_G(m):=\min_{\substack{F\subseteq V(G)\\ |F|=m}}\Delta(G[F])
  \qquad\text{and}\qquad
  a_G(m):=\min_{\substack{F\subseteq V(G)\\ |F|=m}}a(G[F]).$$
  We also set $\Delta_G(0)=a_G(0)=0$. Clearly, $\Delta_G(m)\geq a_G(m)$.
\end{Notation}

The main question is how the quantity $\Delta_G(m)$ depends on $m$ for a given graph $G$. In asymptotic statements, a nonintegral argument such as $\alpha k^n$ may be replaced by either its floor or its ceiling; this choice does not affect any stated asymptotic conclusion.

For Kneser graphs, this question has been answered for a wide range of values of $m$~\cite{KneserCEFL}. In particular, the following result implies the existence of jumps in the parameter $\Delta$.

\begin{Theorem}\cite{KneserCEFL}
  Let $(n, k, s) \in \N^3$ with $n \geq 10000ks^5$. Let $\mathcal{F}\subseteq \binom{[n]}{k}$ be such that $|\mathcal{F}| \geq \binom{n}{k} - \binom{n-s}{k}$ and suppose that $\mathcal{F} \neq \{A \in \binom{[n]}{k}: \,\, A\cap S \neq \emptyset\}$ for all $S \in \binom{[n]}{s}$. Then the subgraph of $K(n, k)$
induced by $\mathcal{F}$ has maximum degree at least

$$(1 - O(\sqrt{s^3 k/n}))\frac{s}{s+1}\cdot \binom{n-k}{k} \cdot \frac{|\mathcal{F}|}{\binom{n}{k}},$$

i.e., there exists $A\in \mathcal{F}$ such that $A$ is disjoint from at least

$$(1 - O(\sqrt{s^3 k/n}))\frac{s}{s+1}\cdot \binom{n-k}{k} \cdot \frac{|\mathcal{F}|}{\binom{n}{k}}$$
 of the sets in $\mathcal{F}$.

\end{Theorem}

One of the graphs studied here---the tensor power of the triangle---behaves similarly to the Kneser graph, and our proof shares several ideas with that of~\cite{KneserCEFL}.

We now recall the definitions of the two graph products considered in this paper.

\begin{Definition}
  The $n$-fold Hamming product of a graph $G$ with itself, denoted $G^{\Box n}$, has as its vertices the length-$n$ sequences of vertices of $G$. Two such sequences are adjacent if and only if they differ in exactly one position and the corresponding vertices in that position are adjacent in $G$; that is,
  $$V(G^{\Box n}) = V(G)^n = \{(v_1,\ldots,v_n) : v_i \in V(G)\},$$
  $$E(G^{\Box n}) = \{\{(v_1,\ldots,v_n), (u_1,\ldots,u_n)\} : \exists i\in [n] \text{ such that } v_j = u_j \text{ for all }j \neq i \text{ and } \{v_i, u_i\} \in E(G)\}.$$
\end{Definition}

\begin{Definition}
  The $n$-fold tensor product of a graph $G$ with itself, denoted $G^{\otimes n}$, has as its vertices the length-$n$ sequences of vertices of $G$. Two such sequences are adjacent if and only if the corresponding vertices are adjacent in every position; that is,
  $$V(G^{\otimes n}) = V(G)^n = \{(v_1,\ldots,v_n) : v_i \in V(G)\},$$
  $$E(G^{\otimes n}) = \{\{(v_1,\ldots,v_n), (u_1,\ldots,u_n)\} : \{v_i, u_i\} \in E(G) \text{ for every }i\in[n]\}.$$
\end{Definition}

\begin{Notation}
  As usual, let $K_k$ denote the complete graph on $k$ vertices. 
\end{Notation}

The principal graphs considered in this paper are $K_3^{\Box n}$ and $K_3^{\otimes n}$. Thus, our main question is:

\begin{Question}
  How do $\Delta_{K_3^{\Box n}}(m)$ and $\Delta_{K_3^{\otimes n}}(m)$ depend on $m$?
\end{Question}

We also provide some generalizations for $K_k^{\Box n}$ and $K_k^{\otimes n}$ for arbitrary $k \geq 3$.

Despite the similarity between $K_3^{\Box n}$ and the hypercube, no Huang-like jump occurs in the former. Indeed, the following folklore statement appears, for example, in~\cite{CayleyGK}.

\begin{Statement}
  The graph $K_3^{\Box n}$ contains an induced matching on $3^{n-1}+1$ vertices.
\end{Statement}

\begin{proof}
  We prove by induction on $n$ the stronger statement that there is an induced matching $M_n$ of the required size that is disjoint from the canonical independent set
  $$A_0^n:=\Big\{x\in\Z_3^n:\sum_i x_i\equiv0\pmod 3\Big\}.$$
  For $n=1$, take $M_1=\{1,2\}$. For the induction step, set
  $$M_n:=M_{n-1}\times\{0\}\;\cup\;A_0^{n-1}\times\{1,2\}.$$
  The first part induces a matching by the induction hypothesis. The second part is the disjoint union of the edges $\{(x,1),(x,2)\}$, $x\in A_0^{n-1}$, because $A_0^{n-1}$ is independent. There are no edges between the two parts, since $M_{n-1}\cap A_0^{n-1}=\emptyset$. Moreover,
  $$|M_n|=(3^{n-2}+1)+2\cdot3^{n-2}=3^{n-1}+1,$$
  and every vector in $M_n$ has nonzero coordinate sum modulo~$3$. Thus $M_n$ has the required properties.
\end{proof}

Potechin and Tsang studied a closely related question for $K_3^{\Box n}$: the maximum number of vertices in an induced subgraph of maximum degree at most~$1$~\cite{PT}. For every $n\geq6$, they constructed such a subgraph on $3^{n-1}+18$ vertices and proved bounds of the form $3^{n-1}+O(1)$ under two additional structural hypotheses. Further bounds on the minimum possible maximum degree of induced subgraphs of the Hamming graphs $K_k^{\Box n}$ appear in~\cite{Hamming}, which primarily concerns low-degree partitions of these graphs.

\paragraph{Our results.} We first determine the exact minimum possible average degree in both graphs.

\begin{Theorem}\label{thm-av}
  \begin{equation*}
  \frac{a_{K_3^{\otimes n}}(m)}{2^n} = \frac{a_{K_3^{\Box n}}(m)}{2n} = \begin{cases}
  0, & m \leq 3^{n-1};\\
  1-\frac{3^{n-1}}{m}, & 3^{n-1} \leq m \leq 2\cdot 3^{n-1};\\
  2-\frac{3^{n}}{m}, & m \geq 2\cdot 3^{n-1}.
  \end{cases}
  \end{equation*}
\end{Theorem}
At first sight, Theorem~\ref{thm-av} is somewhat surprising: after normalization, the two parameters are identical, even though the graphs differ substantially. Claim~\ref{cl3} partly explains this coincidence.

Our main result for the tensor product establishes Huang-like jumps at two points.

\begin{Theorem}\label{thm-jump}
  There exists an absolute constant $c > 0$ such that 
  \begin{enumerate}[a.]
  \item\label{thm:jump0} $\Delta_{K_3^{\otimes n}}(3^{n-1} + 1) \geq c\cdot2^n$ (while $\Delta_{K_3^{\otimes n}}(3^{n-1}) = 0$);
  \item\label{thm:jumphalf} $\Delta_{K_3^{\otimes n}}(2\cdot 3^{n-1} + 1) \geq 2^{n-1} + c\cdot2^n$ (while $\Delta_{K_3^{\otimes n}}(2\cdot 3^{n-1}) = 2^{n-1}$).
  \end{enumerate}
\end{Theorem}

The main tool for proving this theorem is Fourier analysis on the group $\Z_3^n$.

For powers of the Hamming graph, we prove that at several special densities $|F|/3^n$, the minimum possible maximum degree is asymptotically equal to the minimum possible average degree.

\begin{Theorem}\label{thm-ex-hamming}
  For each $\alpha \in \{1 - \frac{1}{2\ell + 1}\}_{\ell \geq 1} \cup \{\frac{2(2^\ell - 1)}{4(2^\ell-1) + 1}\}_{\ell \geq 1} \cup \{\frac{2\ell + 1}{4\ell + 3}\}_{\ell \geq 1} \cup \{\frac{2}{3}\big(1 - \frac{1}{2\ell+1}\big)\}_{\ell \geq 1}$ there is a sequence of sets $F_n \subset \{0,1,2\}^n$ such that $|F_n| = \alpha\cdot 3^n(1-o(1))$ and $\Delta(K_3^{\Box n}[F_n]) = a_{K_3^{\Box n}}(\alpha\cdot 3^{n})(1+o(1))$.
\end{Theorem}

\paragraph{Structure of the paper.}
Section~\ref{sec-average} proves Theorem~\ref{thm-av}, which gives the first natural lower bound on the minimum possible maximum degree. Section~\ref{sec:tensor} treats the tensor product, while Section~\ref{sec:hamming} treats the Hamming product. In Subsection~\ref{sec:proof-jump}, we establish Huang-like behavior for the tensor product and for its large canonical bipartite subgraphs; in particular, we prove Theorem~\ref{thm-jump}. Subsection~\ref{ssec:ex-tensor} gives optimal constructions and upper bounds. Subsections~\ref{ssec:no-big-perfect} and~\ref{ssec:no-small-perfect} identify ranges in which the minimum possible maximum and average degrees cannot coincide, whereas Subsection~\ref{ssec:ex-tensor} gives several densities at which they do coincide. Section~\ref{sec:hamming} presents the constructions asserted in Theorem~\ref{thm-ex-hamming}. For the Hamming product, our only general lower bound is the average-degree bound. Subsection~\ref{subsec-hamming-codes} gives exactly optimal constructions derived from well-known Hamming codes, and the remainder of Section~\ref{sec:hamming} develops additional constructions with reasonably small, though not always optimal, maximum degree. Section~\ref{sec:open-q} discusses open problems and future directions. Appendix~\ref{sec-gen-k} extends several results to $K_k^{\otimes n}$ and $K_k^{\Box n}$.

\section{Exact answer for average degree}\label{sec-average}

\begin{Claim}\label{cl1}
  Let $G$ be any $d$-regular graph. Then
  $$\frac{a_G(m)}{d}\geq 2-\frac{|V(G)|}{m}$$
  for every $m\geq |V(G)|/2$.
\end{Claim}

\begin{proof}
  Let $E(F)$ denote the set of edges with both endpoints in $F$, let $E(F,V\setminus F)$ denote the set of edges between $F$ and $V\setminus F$, and let $E(V\setminus F)$ denote the set of edges with both endpoints in $V\setminus F$. Then
$$2|E(F)|+|E(F,V\setminus F)|=d|F|,$$
$$2|E(V\setminus F)|+|E(F,V\setminus F)|=d(|V|-|F|).$$
Subtracting the second identity from the first gives
$$|E(F)|=d\Big(|F|-\frac{|V|}{2}\Big)+|E(V\setminus F)|\geq d\Big(|F|-\frac{|V|}{2}\Big),$$
which proves the claim.
\end{proof}

\begin{Claim}\label{cl2}
  Suppose that $G$ is a $D$-regular edge-transitive graph and that $H$ is a spanning $d$-regular subgraph of $G$. Then, for every $m$, we have $\frac{a_G(m)}{D} \geq \frac{a_H(m)}{d}$.
\end{Claim}

\begin{proof}
  Let $F\subseteq V(G)$ have size $m$, and let $\sigma$ be a uniformly random automorphism of $G$.
  Then $$\frac{ma_H(m)}{2} \leq \mathbb{E}[|E(H[\sigma(F)])|] = \sum_{e\in E(G[F])}\Pp[\sigma e\in E(H)] = |E(G[F])|\frac{|E(H)|}{|E(G)|} = |E(G[F])|\frac{d}{D},$$
  where we used edge-transitivity and linearity of expectation. Consequently,
  $$\frac{a_G(m)}{D} \geq \frac{a_H(m)}{d}.$$
\end{proof}

\begin{Claim}\label{cl3}
  $K_3^{\Box n}$ is isomorphic to a subgraph of $K_3^{\otimes n}$.
\end{Claim}
The Hamming product is not a subgraph of the tensor product under the standard labeling of the common vertex set $\{0,1,2\}^n$. Rather, it becomes a subgraph after the relabeling described below.

\begin{proof}
  Regard the vertices of both graphs as vectors in $\Z_3^n$. Multiplication by any nonsingular matrix over $\Z_3$ with no zero entries gives the required embedding. For example, take
  $$M = \begin{pmatrix}
  1&1&1&\cdots&1\\
  1&2&1&\cdots&1\\
  1&1&2&\cdots&1\\
  \vdots&\vdots&\vdots&\ddots&\vdots\\
  1&1&1&\cdots&2
\end{pmatrix}.$$
The matrix $M$ is nonsingular: subtracting the first row from every other row yields an upper-triangular matrix with nonzero diagonal entries.

Multiplication by such a matrix indeed sends the edges of $K_{3}^{\Box n}$ to the edges of $K_{3}^{\otimes n}$: two vertices $x$ and $y$ are adjacent in $K_{3}^{\Box n}$ if $x - y = \pm e_i$ for some $i \in [n]$, where $\{e_i\}_i$ is the standard basis of $\Z_3^n$. Then $Mx-My$ is, up to sign, a column of $M$ and therefore has no zero coordinates, which implies that $Mx$ and $My$ are adjacent in $K_3^{\otimes n}$.
\end{proof}
\begin{Claim}\label{cl4}
  $$\alpha(K_3^{\otimes n}) = \alpha(K_3^{\Box n}) = 3^{n-1}.$$
\end{Claim}

\begin{proof}
  By Claim~\ref{cl3}, $\alpha(K_3^{\otimes n}) \leq \alpha(K_3^{\Box n})$. The tensor product has an independent set $\{0\}\times \{0,1,2\}^{n-1}$ of size $3^{n-1}$. Suppose that $F \subset V(K_3^{\Box n})$ is such that $|F| > 3^{n-1}$. Then by the pigeonhole principle $F$ contains two sequences with the same $(n-1)$-prefix, and these two vertices are adjacent.
\end{proof}

We now prove Theorem~\ref{thm-av}, namely, that 

\begin{equation*}
  \frac{a_{K_3^{\otimes n}}(m)}{2^n} = \frac{a_{K_3^{\Box n}}(m)}{2n} = \begin{cases}
  0, & m \leq 3^{n-1};\\
  1-\frac{3^{n-1}}{m}, & 3^{n-1} \leq m \leq 2\cdot 3^{n-1};\\
  2-\frac{3^{n}}{m}, & m \geq 2\cdot 3^{n-1}.
  \end{cases}
\end{equation*}

\begin{proof}
By Claims~\ref{cl2} and~\ref{cl3},
$$\frac{a_{K_3^{\otimes n}}(m)}{2^n} \geq \frac{a_{K_3^{\Box n}}(m)}{2n}.$$
Thus, in each range it suffices to give an upper-bound construction for the tensor product and a matching lower bound for the Hamming product.

\begin{enumerate}
\item If $m\leq 3^{n-1}$, the result follows from Claim~\ref{cl4}.

\item Suppose that $3^{n-1}\leq m\leq 2\cdot3^{n-1}$.
  \begin{itemize}
  \item For the lower bound, let $F\subseteq\Z_3^n$ and write $E:=E(K_3^{\Box n}[F])$. Let $e_1,\ldots,e_n$ be the standard basis of $\Z_3^n$, and set
  $$\mathcal E:=\{e_1,\ldots,e_n,2e_1,\ldots,2e_n\}.$$
  Then
  $$\sum_{e\in\mathcal E}|F\cap(F+e)|
  =\sum_{x\in F}\deg_{K_3^{\Box n}[F]}(x)=2|E|.$$
  Hence, for some $e\in\mathcal E$,
  $$|F\cap(F+e)|\leq\frac{|E|}{n}.$$
  Translation invariance gives
  $$|(F+e)\cap(F+2e)|=|(F+2e)\cap F|=|F\cap(F+e)|\leq\frac{|E|}{n}.$$
  Therefore,
  \begin{multline*}
  3^n\geq|F\cup(F+e)\cup(F+2e)|\\
  \geq 3|F|-|(F+e)\cap(F+2e)|-|(F+2e)\cap F|-|F\cap(F+e)|
  \geq 3|F|-\frac{3|E|}{n}.
  \end{multline*}
  Thus
  $$|E|\geq n(|F|-3^{n-1})$$
  and consequently
  $$a_{K_3^{\Box n}}(m)\geq 2n\left(1-\frac{3^{n-1}}{m}\right).$$

  \item For the upper bound, let
  $$F=\{0\}\times\Z_3^{n-1}\;\cup\;\{1\}\times X\subseteq V(K_3^{\otimes n}),$$
  where $X\subseteq\Z_3^{n-1}$ has size $m-3^{n-1}$. Then
  $$|E(K_3^{\otimes n}[F])|=2^{n-1}|X|,$$
  so
  $$\frac{2|E(K_3^{\otimes n}[F])|}{m}=2^n\left(1-\frac{3^{n-1}}{m}\right).$$
  \end{itemize}

\item Suppose that $m\geq2\cdot3^{n-1}$.
  \begin{itemize}
  \item The lower bound
  $$\frac{a_{K_3^{\Box n}}(m)}{2n}\geq2-\frac{3^n}{m}$$
  follows from Claim~\ref{cl1}.
  \item For the upper bound, let
  $$F=\{0,1\}\times\Z_3^{n-1}\;\cup\;\{2\}\times X\subseteq V(K_3^{\otimes n}),$$
  where $X\subseteq\Z_3^{n-1}$ has size $m-2\cdot3^{n-1}$. Then
  $$|E(K_3^{\otimes n}[F])|=2^n|X|+2^{n-1}3^{n-1}=2^{n-1}(2m-3^n),$$
  and hence
  $$\frac{2|E(K_3^{\otimes n}[F])|}{m}=2^n\left(2-\frac{3^n}{m}\right).$$
  \end{itemize}
\end{enumerate}
\end{proof}

The normalized minimum average degree depends only on the density $|F|/3^n$. Moreover, a construction in one dimension can be lifted to higher dimensions without increasing the corresponding normalized maximum degree, as follows.

\begin{Claim}
  $$\frac{\Delta_{K_{3}^{\otimes(n+1)}}(3m)}{2^{n+1}} \leq \frac{\Delta_{K_{3}^{\otimes n}}(m)}{2^{n}}$$
  and 
  $$\frac{\Delta_{K_{3}^{\Box tn}}(3^{n(t-1)}m)}{2nt} \leq \frac{\Delta_{K_{3}^{\Box n}}(m)}{2n}$$
  for any integer $t \geq 1$.
\end{Claim}
\begin{proof}
  For the first inequality, given $F\subseteq V(K_3^{\otimes n})$, take $F\times\{0,1,2\}$.

  For the second inequality, given $F\subseteq V(K_3^{\Box n})$, take
  $$\left\{x\in\{0,1,2\}^{nt}:\left(\sum_{i=1}^{t}x_i,\sum_{i=t+1}^{2t}x_i,\ldots,\sum_{i=(n-1)t+1}^{nt}x_i\right)\in F\right\},$$
  where all sums are taken modulo~$3$.
\end{proof}

\section{Minimum possible maximum degree for the tensor product}\label{sec:tensor}

\subsection{Proof of Theorem~\ref{thm-jump}}\label{sec:proof-jump}

\subsubsection{Fourier decomposition for $\Z_3^n$ and the spectrum of the tensor product}\label{ssec:fourier}

To prove Theorem~\ref{thm-jump}, we use the Fourier basis on $\Z_3^n$.

\begin{Notation}
  For $x,y\in\Z_3^n$, define $u_y(x):=\omega_3^{x\cdot y}$, where $\omega_3$ is a primitive third root of unity. We equip $L_2(\Z_3^n)$ with the normalized inner product
  $$\langle f,g\rangle:=3^{-n}\sum_{x\in\Z_3^n}f(x)\overline{g(x)}.$$
  The functions $\{u_y\}_{y\in\Z_3^n}$ form an orthonormal basis, called the Fourier basis. The unique expansion of a function in this basis is its Fourier decomposition.
\end{Notation}

\begin{Notation}
  For $y \in \Z_3^n$ denote $|y| := \#\{i\in [n] : y_i \neq 0\}$.
\end{Notation}

We identify $\Co^{V(K_3^{\otimes n})}$ with $L_2(\Z_3^n)$ in a natural way: $f_v = f(v)$. In particular, we can treat the adjacency matrix $B_n := \begin{pmatrix}
  0&1&1\\
  1&0&1\\
  1&1&0
\end{pmatrix}^{\otimes n}$ of $K_3^{\otimes n}$ as acting on $L_2(\Z_3^n)$.

\begin{Claim}
  The Fourier basis $\{u_y\}_{y \in \Z_3^n}$ is an eigenbasis of $B_n$, where $u_y$ corresponds to the eigenvalue $\lambda_y = (-1)^{|y|}\cdot 2^{n - |y|}$.
\end{Claim}

\begin{proof}
\begin{multline*}
  (B_nu_y)(x) = \sum_{x' : \{x,x'\} \in E(K_3^{\otimes n})} u_y(x') = \sum_{x'' \in \{1,2\}^n} u_y(x + x'') = \omega_3^{x\cdot y}\prod_{i=1}^n\sum_{x''_i \in \{1,2\}}\omega_3^{x''_i y_i} =\\= \omega_3^{x\cdot y}\prod_{i=1}^n(\omega_3^{y_i}+\omega_3^{-y_i}) = (-1)^{|y|}\cdot 2^{n - |y|}\omega_3^{x\cdot y} = (-1)^{|y|}\cdot 2^{n - |y|}u_y(x).
\end{multline*}
\end{proof}

We shall use the following lemma from~\cite{ADFS}.

\begin{Lemma}[Lemma 2.4 in \cite{ADFS}]\label{lemm}
  For every $k\geq2$ there exists a constant $K=K(k)$ with the following property. Let $f:\Z_k^n\to\{0,1\}$ satisfy $\Pr[f=1]=\alpha$ and
  $$\sum_{|y|>1}|\hat f(y)|^2=\epsilon,$$
  where $0<\epsilon<\alpha-\alpha^2$. Then there exists a Boolean function $g:\Z_k^n\to\{0,1\}$ depending on at most one coordinate such that
  $$\|f-g\|_2^2<\frac{K\epsilon}{\alpha-\alpha^2-\epsilon}.$$
  \end{Lemma}

The spectrum of the tensor product consists of the eigenvalues $\lambda_k=(-1)^k2^{n-k}$, each with multiplicity $\binom{n}{k}2^k$, for $k=0,\ldots,n$. Hence it is a $(3^n,2^n,2^{n-1})$-graph. We will use the following bipartite consequence of spectral expansion, proved in~\cite{CayleyGK}.

\begin{Theorem}[Theorem 7.2 in \cite{CayleyGK}]
  If $G$ is the tensor product of $K_2$ and an $(n,d,\lambda)$-graph, then $\Delta_G(|V(G)|/2+1)>(d-\lambda)/2$.
\end{Theorem}

We shall apply this result to some canonical bipartite subgraphs of the tensor product in the next section. 

\subsubsection{The jump for the bipartite subgraph at 0}\label{bipart_jump}
The tripartite tensor product under study is the union of three interacting bipartite graphs. A useful first step toward understanding the minimum possible maximum degree in the whole graph is therefore to analyze these bipartite components. We will later see that every known optimal example at small densities $m/3^n$ is contained in $\{0,1\}\times\{0,1,2\}^{n-1}$.

Consider the induced subgraph $G_n'$ of the tensor product on $\{0,1\} \times \{0,1,2\}^{n-1}$. It is the tensor product of $K_2$ and $K_3^{\otimes(n-1)}$. Since the latter is a $(3^{n-1},2^{n-1},2^{n-2})$-graph, Theorem~7.2 of~\cite{CayleyGK} implies that $\Delta_{G'_n}(3^{n-1} +1) > \frac{2^{n-1}-2^{n-2}}{2} = \frac{2^{n}}{8}$. The constant $\frac{1}{8}$ can be slightly improved:

\begin{Claim}\label{cl:better-bound-bipart}
  $\Delta_{G'_n}(3^{n-1}+1)>0.52\cdot2^{n-2}$.
\end{Claim}

The first part of the proof follows the argument of Theorem~7.2 in~\cite{CayleyGK}, together with the expander mixing lemma, to show that if the maximum degree is close to the stated bound, then the set must meet each side of the bipartition in approximately $3^{n-1}/2$ vertices. Claim~\ref{cl:no-conc} then bounds the first nonconstant Fourier level of the characteristic function of each part. We use this estimate to count the edges of the induced subgraph more precisely.

\begin{Claim}\label{cl:no-conc}
  Let $f\in L_2(\Z_3^n)$ be Boolean. Let $f = \sum_{y\in \Z_3^n} \hat f(y)u_y$ be its Fourier decomposition. Suppose that $\alpha := \hat{f}(0) \in [\frac13, \frac23]$. Denote $\gamma := \sum_{i = 1}^{n}|\hat f(e_i)|^2 + \sum_{i = 1}^{n}|\hat f(-e_i)|^2$. Then
  $$\frac{3}{4}\gamma \leq \frac{1}{3} - \frac{3}{4}\alpha + \frac{3}{4}\alpha^2.$$
\end{Claim}
\begin{proof}
  Let $F = \{x : f(x) = 1\}$. Note that by definition of the minimum possible average degree $\alpha a_{K_3^{\otimes n}}(|F|) \leq \frac{f^TB_{n}f}{3^{n}} \leq 2^{n}\alpha^2 - 2^{n-1}\gamma + 2^{n-2}(\alpha - \alpha^2 - \gamma)$. For $\frac{1}{3} \leq \alpha \leq \frac{2}{3}$ we have $a_{K_3^{\otimes n}}(|F|) = 2^{n}(1 - \frac{1}{3\alpha})$. Substituting this into the previous expression and rearranging we get $\frac{3}{4}\gamma \leq \frac{1}{3} - \frac{3}{4}\alpha + \frac{3}{4}\alpha^2$, as desired.
\end{proof}

\begin{proof}[Proof of Claim~\ref{cl:better-bound-bipart}]
  Assume the contrary: suppose there exists $F \subset \{0,1\} \times \{0,1,2\}^{n-1}$ of size $3^{n-1} +1$ such that the induced subgraph has maximum degree at most $d = 0.52 \cdot 2^{n-2}$. Delete one element from $F$ so that $F$ still intersects both parts of $G_n'$. Let $f_0 \in L_2(\{0\} \times \Z_3^{n-1})$ be the characteristic function of $F_0 := F\cap (\{0\} \times \Z_3^{n-1})$ and $f_1 \in L_2(\{1\} \times \Z_3^{n-1})$ be the characteristic function of $F_1 = F\cap (\{1\} \times \Z_3^{n-1})$. Then $e(F_0, F_1) = f_0^TB_{n-1}f_1$. Let $f_0 = \sum_{y\in \Z_3^{n-1}}\hat f_0(y)u_y$ and $f_1 = \sum_{y\in \Z_3^{n-1}}\hat f_1(y)u_y$ be the Fourier decompositions of these functions. Let $f_0' = \sum_{y\neq 0}\hat f_0(y)u_y$ and $f_1' = \sum_{y\neq 0}\hat f_1(y)u_y$. Let $\alpha_0 := \hat f_0(0) = \frac{|F_0|}{3^{n-1}}$ and $\alpha_1 := \hat f_1(0) = \frac{|F_1|}{3^{n-1}}$. Notice that $\alpha_0 + \alpha_1 = 1$. Assume without loss of generality that $\alpha_0 \geq \alpha_1$. We will now prove that both $\alpha_0$ and $\alpha_1$ should be close to $\frac{1}{2}$, namely,

  \begin{Claim}\label{cl-half}
  $0.5 \leq \alpha_0 \leq 0.52$.
  \end{Claim}
  \begin{proof}
  The first inequality is obvious from $\alpha_0 + \alpha_1 = 1$ and $\alpha_0 \geq \alpha_1$. 

  Towards the second inequality we shall express the number of edges in $F$ via the Fourier decompositions of the characteristic functions of the parts, namely, 
  we have 
  
  $$d\alpha_1 \geq \frac{e(F_1, F_0)}{3^{n-1}} = 2^{n-1}\alpha_0 \alpha_1 + \langle f_1',B_{n-1}f_0'\rangle \geq 2^{n-1}\alpha_0 \alpha_1 - 2^{n-2}\sqrt{\alpha_0- \alpha_0^2} \sqrt{\alpha_1- \alpha_1^2} = 2^{n-2}\alpha_0\alpha_1,$$
  where the last inequality follows from Cauchy--Schwarz, and the last equality from $\alpha_0 + \alpha_1 = 1$. 
  Since $\alpha_1 > 0$, this implies $d \geq 2^{n-2}\alpha_0$, and so $\alpha_0 \leq \frac{d}{2^{n-2}} = 0.52$, as desired.
  \end{proof}

  Now let $\gamma_0 := \sum_{i = 1}^{n-1}|\hat f_0(e_i)|^2 + \sum_{i = 1}^{n-1}|\hat f_0(-e_i)|^2$. Then by Claim~\ref{cl:no-conc} we get $\frac{3}{4}\gamma_0 \leq \frac{1}{3} - \frac{3}{4}\alpha_0 + \frac{3}{4}\alpha_0^2$.

  Then $||B_{n-1}f_0'||^2 \leq 2^{2(n-2)}\gamma_0 + 2^{2(n-3)}(\alpha_0 - \alpha_0^2 -\gamma_0) = 2^{2(n-2)}(\frac{3}{4}\gamma_0 + \frac{1}{4}(\alpha_0 - \alpha_0^2)) \leq 2^{2(n-2)}(\frac{1}{3} - \frac{1}{2}(\alpha_0 - \alpha_0^2))$.

  Hence, once again estimating the number of edges and using the Cauchy--Schwarz inequality, we get 

  \begin{multline*}
  \alpha_1 d \geq 2^{n-1}\alpha_0 \alpha_1 + \langle f_1',B_{n-1}f_0'\rangle \geq 2^{n-1}\alpha_0 \alpha_1 - ||f_1'|| \cdot ||B_{n-1}f_0'|| \geq\\\geq 2^{n-1}\alpha_0\alpha_1 - 2^{n-2}\sqrt{\alpha_1 - \alpha_1^2}\sqrt{\frac{1}{3} - \frac{1}{2}(\alpha_0 - \alpha_0^2)}.
  \end{multline*}

  Consequently, since the function $x - x^2$ is decreasing for $x \geq 0.5$, and $\frac{x}{1-x}$ is increasing for $x\in (0,1)$, using Claim~\ref{cl-half} we conclude that
  
  $$\frac{d}{2^{n-2}} \geq 2\alpha_0 - \sqrt{\frac{\alpha_0}{1-\alpha_0}}\sqrt{\frac{1}{3} - \frac{1}{2}(\alpha_0 - \alpha_0^2)} \geq 1 - \sqrt{\frac{0.52}{1-0.52}}\sqrt{\frac{1}{3} - \frac{1}{2}(0.52 - 0.52^2)} > 0.524,$$
  contradiction.
  
\end{proof}

\subsubsection{The jump at 0 for the whole graph}\label{gap1}

We now prove part~\ref{thm:jump0} of Theorem~\ref{thm-jump}: there exists an absolute constant $c>0$ such that every $F\subseteq\Z_3^n$ with $|F|>3^{n-1}$ satisfies $\Delta(K_3^{\otimes n}[F])\geq c2^n$. Equivalently, $\Delta_{K_3^{\otimes n}}(3^{n-1}+1)\geq c2^n$.

\begin{proof}
Let $F\subseteq\Z_3^n$ induce a subgraph of maximum degree at most $d$. By deleting vertices if necessary, we may assume that $|F|=3^{n-1}+1$. Set $\alpha:=|F|/3^n$ and $\varepsilon:=d/2^n$. It suffices to consider $\varepsilon<1/1000$.

Let $\un_F = \sum_{y\in \Z_3^n}\hat f(y) u_y$ be the Fourier decomposition of the characteristic function of $F$. Define $P_\ell(\un_F):=\sum_{y:|y|=\ell}\hat f(y)u_y$. Let $P_{\geq2}(\un_F):=\sum_{\ell \geq 2}P_\ell(\un_F)$.

We distinguish two cases according to the size of $||P_{\geq 2}(\un_F)||$. If this projection is large, then the induced subgraph on $F$ has large average degree. If it is small, Lemma~\ref{lemm} implies that $\un_F$ is close to a function of one coordinate, and we derive the desired degree bound from that structure.

\textbf{Case 1: $||P_{\geq 2}(\un_F)||_2^2 > \frac{8}{3}\alpha\varepsilon$.}

We have 
\begin{multline*}
  2|E(K_{3}^{\otimes n}[F])| = \un_F^TB_n \un_F = 3^n\sum_{k=0}^n \lambda_k||P_k(\un_F)||_2^2\geq \\ \geq 2^n\cdot 3^n\cdot ||P_0(\un_F)||^2_2 - 2^{n-1}\cdot 3^n\cdot ||P_1(\un_F)||^2_2 - 2^{n-3}\cdot 3^n\cdot ||P_{\geq 2}(\un_F)||^2_2 =\\= 3^n(2^n\alpha^2 - 2^{n-1}(\alpha - \alpha^2 - ||P_{\geq 2}(\un_F)||^2_2) - 2^{n-3}||P_{\geq 2}(\un_F)||^2_2) =\\= 3^n\cdot 2^{n-1}(3\alpha^2 - \alpha) + 3^n\cdot 2^{n-3}\cdot 3||P_{\geq 2}(\un_F)||^2_2 > 3^n\cdot 2^n \cdot \alpha\varepsilon = d|F|,
\end{multline*}
and so the average degree of $K_{3}^{\otimes n}[F]$ is greater than $d$, contradiction.

\vspace{2mm}

\textbf{Case 2: $||P_{\geq 2}(\un_F)||_2^2 \leq \frac{8}{3}\alpha\varepsilon$.}

We first apply Lemma~\ref{lemm} with $k = 3$.

For the characteristic function of $F$ it produces a Boolean function $g$ depending on at most one coordinate, such that 

\begin{multline*}
  ||\un_F-g||_2^2 < \frac{K}{\alpha - \alpha^2 - ||P_{\geq 2}(\un_F)||^2_2}||P_{\geq 2}(\un_F)||^2_2 \leq\\ \leq \frac{K}{\alpha - \alpha^2 - \frac83\alpha\varepsilon}\cdot\frac83\alpha\varepsilon = \frac{K}{1 - \alpha - \frac{8}{3}\varepsilon}\cdot \frac{8}{3}\varepsilon.
\end{multline*}

For $n$ sufficiently large we have $\alpha \leq \frac{2}{3}$, and so $||\un_F-g||_2^2 < \frac{8\varepsilon K}{1 - 8\varepsilon}$. Without loss of generality, assume that $g$ depends only on the first coordinate. 

Let $F_i := F\cap \{x_1 = i\}$ for $i = 0, 1, 2$. Let $\alpha_i := \frac{|F_i|}{3^{n-1}}$. We can assume without loss of generality that $\alpha_0 \geq \alpha_1 \geq \alpha_2$, then $\alpha_0 \geq \alpha \geq \frac{1}{3}$ and $\alpha_2 \leq \alpha \leq \frac{2}{3}$. Thus if $g(x_1 = 0) = 0$ or $g(x_1 = 2) = 1$ we have $||\un_F-g||_2^2 \geq \frac19$. In the remaining case $g(x_1 = 0) = 1$ and $g(x_1 = 2) = 0$, and so $||\un_F-g||_2^2 \geq \frac{1-\alpha_0}{3} + \frac{\alpha_2}{3}$. Now we will estimate the right hand side using the expander mixing lemma.

Since $|F| > 3^{n-1}$, $\alpha_1 > 0$. Delete one vertex from $F$ getting a set $F'$ such that $\alpha'_0 + \alpha'_1 + \alpha'_2 = 1$ and still $\alpha'_0 \geq \alpha'_1 \geq \alpha'_2$ and $\alpha'_1 > 0$. (Here and later in similar contexts $\alpha_i'$ and $F'_i$ mean the same quantities for $F'$ as $\alpha_i$ and $F_i$ do for $F$.) Consider once again $G'_n$, the bipartite subgraph induced on $\{0,1\} \times \Z_3^{n-1}$. This graph is the tensor product of $K_3^{\otimes (n-1)}$ and an edge. Thus by the expander mixing lemma for $K_3^{\otimes (n-1)}$ 
  $$\bigg|e(F'_0, F'_1) - \frac{2^{n-1}|F'_0||F'_1|}{3^{n-1}}\bigg| \leq 2^{n-2}\sqrt{|F'_0||F'_1|(1-\alpha'_0)(1-\alpha'_1)},$$
  and hence
  $$d\cdot \alpha'_1 \geq \frac{e(F'_0,F'_1)}{3^{n-1}} \geq 2^{n-1}\alpha'_0\alpha'_1 - 2^{n-2}\sqrt{\alpha'_0\alpha'_1(1-\alpha'_0)(1-\alpha'_1)}.$$
  Then 
  $$4\varepsilon \geq \alpha'_0\bigg(2 - \sqrt{\Big(1+\frac{\alpha'_2}{\alpha'_1}\Big)\Big(1+\frac{\alpha'_2}{\alpha'_0}\Big)}\bigg) \geq \alpha'_0\bigg(2 - \sqrt{2\Big(1+\frac{\alpha'_2}{\alpha'_0}\Big)}\bigg).$$
  Since $\varepsilon \leq \frac{1}{6}$ this implies
  $$\alpha'_0 - \alpha'_2 \leq 8\varepsilon - \frac{8\varepsilon^2}{\alpha'_0} \leq 8\varepsilon - 8\varepsilon^2.$$

  Since we deleted only one vertex, for the actual set $F$ we get $\alpha_0 - \alpha_2 \leq 8\varepsilon$ for $n$ sufficiently large.

Consequently, $||\un_F-g||_2^2 \geq \frac{1}{3} - \frac{8}{3}\varepsilon \geq \frac{1}{9}$ as $\varepsilon \leq \frac{1}{12}$.

Overall, $$\frac{1}{9} \leq ||\un_F-g||_2^2< \frac{8\varepsilon K}{1 - 8\varepsilon}.$$

We get $\frac{d}{2^n} = \varepsilon > \frac{1}{8(9K+1)}$.

\end{proof}

\paragraph{A curious fact: absence of an optimal example.}

From the statement above we see that $\epsilon := \inf_n \frac{\min_{|F|>3^{n-1}} \Delta(K_3^{\otimes n}[F])}{2^n} > 0$.

\begin{Claim}
  This infimum cannot be achieved.
\end{Claim}

\begin{proof}
  Suppose, to the contrary, that there exist $n$ and $F$ such that $|F|>3^{n-1}$ and $\Delta(K_3^{\otimes n}[F])=\epsilon 2^n$. For any $m\in\Z_{\geq2}$, let
  $$Q_m:=\{0,1,2\}^m\setminus\{0,1\}^m.$$
  Every vector in $\Z_3^m$ has a neighbour in $\{0,1\}^m$, and hence
  $$\Delta(K_3^{\otimes m}[Q_m])\leq2^m-1.$$
  Therefore,
  $$\Delta(K_3^{\otimes(n+m)}[F\times Q_m])
  \leq(2^m-1)\epsilon2^n
  =(1-2^{-m})\epsilon2^{n+m}.$$
  Since $(1-2^{-m})\epsilon<\epsilon$, the definition of $\epsilon$ implies that $|F\times Q_m|\leq3^{n+m-1}$. Thus
  $$|F|(3^m-2^m)\leq3^{n+m-1},$$
  or equivalently,
  $$|F|\left(1-\left(\frac23\right)^m\right)\leq3^{n-1}.$$
  Letting $m\to\infty$ gives $|F|\leq3^{n-1}$, a contradiction.
\end{proof}

\subsubsection{The jump at $\Delta = 2^{n-1}$}\label{ssec:jump-2n-1}

Recall that $\Delta_{K_3^{\otimes n}}(2\cdot 3^{n-1}) = 2^{n-1}$.

We now prove part~\ref{thm:jumphalf} of Theorem~\ref{thm-jump}; namely,
there exists an absolute constant $c > 0$ such that for any $F \subset \Z_3^n$ such that $|F| > 2\cdot 3^{n-1}$ we have $\Delta(K_3^{\otimes n}[F]) \geq 2^{n-1} + c\cdot2^n$, i.e., $\Delta_{K_3^{\otimes n}}(2\cdot 3^{n-1} + 1) \geq 2^{n-1} + c\cdot 2^n$.

The proof is analogous to that of the first part, but we include it for completeness. 

\begin{proof}

Suppose that the induced subgraph of $K_3^{\otimes n}$ on the set $F$ has maximum degree less than or equal to $2^{n-1} + d$ and $|F| = 2\cdot 3^{n-1} + 1$. 

Use the Fourier notation from Subsection~\ref{gap1}: write $\un_F=\sum_{y\in\Z_3^n}\hat f(y)u_y$, set $\alpha:=|F|/3^n$ and $\varepsilon:=d/2^n$, and let $P_j(\un_F)$ denote the projection onto level $j$. Again it suffices to consider $\varepsilon<1/1000$. We split into two cases according to the size of $\|P_{\geq2}(\un_F)\|_2^2$: a large high-level component forces the average degree to be too large, whereas a small one allows us to apply Lemma~\ref{lemm}.

\textbf{Case 1: $||P_{\geq 2}(\un_F)||_2^2 > \frac{8}{3}\alpha\varepsilon$.}

Similarly to the first part of the theorem, we have
\begin{multline*}
  2|E(K_{3}^{\otimes n}[F])| = \un_F^TB_n \un_F \geq 2^n\cdot 3^n\cdot ||P_0(\un_F)||^2_2 - 2^{n-1}\cdot 3^n\cdot ||P_1(\un_F)||^2_2 - 2^{n-3}\cdot 3^n\cdot ||P_{\geq 2}(\un_F)||^2_2 =\\= 3^n(2^n\alpha^2 - 2^{n-1}(\alpha - \alpha^2 - ||P_{\geq 2}(\un_F)||^2_2) - 2^{n-3}||P_{\geq 2}(\un_F)||^2_2) =\\= 3^n\cdot 2^{n-1}(3\alpha^2 - \alpha) + 3^n\cdot 2^{n-3}\cdot 3||P_{\geq 2}(\un_F)||^2_2 >\\> 3^n\cdot 2^{n-1}(3\alpha^2 - 2\alpha) + 3^n\cdot 2^{n-1}\cdot\alpha + 3^n\cdot 2^n \cdot \alpha\varepsilon \geq (2^{n-1} + d)|F|,
\end{multline*}
where the last inequality holds because $\alpha \geq \frac23$, and so the average degree of $K_{3}^{\otimes n}[F]$ is strictly greater than the claimed maximum degree, contradiction.

\textbf{Case 2: $||P_{\geq 2}(\un_F)||_2^2 \leq \frac{8}{3}\alpha\varepsilon$.}

In this case we apply Lemma~\ref{lemm} with $k=3$ and obtain that there exists a Boolean function $g$ depending on at most one coordinate such that 

$$||\un_F-g||_2^2 \leq \frac{K}{1 - \alpha - \frac{8}{3}\varepsilon}\cdot \frac{8}{3}\varepsilon.$$

For sufficiently large $n$ we have $\alpha \leq \frac{2}{3} + \frac{2}{3}\varepsilon$, and so $$||\un_F-g||_2^2 \leq \frac{K}{1 -10\varepsilon}\cdot 8\varepsilon \leq 9K\varepsilon.$$

We can assume without loss of generality that $g$ depends only on the first coordinate.

Let $F_i$ and $\alpha_i$ be as in the previous subsection, ordered so that $|F_0|\geq|F_1|\geq|F_2|$. Since $\alpha_2>0$, we can delete one vertex to obtain a set $F'$ satisfying $\alpha'_0\geq\alpha'_1\geq\alpha'_2>0$ and $\alpha'_0+\alpha'_1+\alpha'_2=2$. Set $F'_{01}:=F'_0\cup F'_1$.

Consider the bipartite subgraph (not induced) $G''_n$ of $K_3^{\otimes n}$ with one part $\{0,1\}\times \{0,1,2\}^{n-1}$ and the other part $\{2\} \times \{0,1,2\}^{n-1}$ formed by all the edges between the parts. The matrix of this bipartite graph is $\begin{pmatrix}
  B_{n-1}\\
  B_{n-1}
\end{pmatrix}$. Its singular values are $\sqrt{2}\,2^j$, $j=0,\ldots,n-1$. The largest, $\sigma_1=\sqrt{2}\,2^{n-1}$, has multiplicity one, and the second largest is $\sigma_2=\sqrt{2}\,2^{n-2}$. The bipartite expander mixing lemma therefore gives

$$\bigg\lvert \frac{e(F_{01}', F_2')}{|E(G''_n)|} - \frac{\alpha'_0+\alpha'_1}{2}\alpha'_2\bigg\rvert \leq \frac{\sigma_2}{\sigma_1}\sqrt{\frac{\alpha'_0+\alpha'_1}{2}\Big(1-\frac{\alpha'_0+\alpha'_1}{2}\Big)\alpha'_2(1-\alpha'_2)}.$$

This implies $$\alpha'_2\Big(\frac{1}{2} + \varepsilon\Big) \geq \frac{e(F_{01}', F_2')}{|E(G''_n)|} \geq \frac{\alpha'_0 + \alpha'_1}{2}\alpha'_2 - \frac{1}{2}\sqrt{\frac{\alpha'_0+\alpha'_1}{2}\Big(1-\frac{\alpha'_0+\alpha'_1}{2}\Big)\alpha'_2(1-\alpha'_2)}.$$

Using $\alpha_0' + \alpha_1' = 2 - \alpha_2'$ and $\alpha_2' > 0$, we obtain $$\sqrt{(2-\alpha_2')(1 - \alpha_2')} \geq 2(1-\alpha_2') - 4\varepsilon.$$

Then either $2(1 - \alpha_2') - 4\varepsilon < 0$ and so $\alpha_2' > 1 - 2\varepsilon$, or $(2 - \alpha_2')(1-\alpha_2') \geq (2(1-\alpha_2') - 4\varepsilon)^2 \geq 4(1 - \alpha_2')(1 - \alpha_2' - 4\varepsilon)$, and so $\alpha_2' \geq \frac{2}{3} - \frac{16}{3}\varepsilon$. So $\alpha_2' \geq \frac{2}{3} - \frac{16}{3}\varepsilon$ anyway. Since we only deleted vertices to obtain $F'$, we also have $\alpha_2 \geq \frac{2}{3} - \frac{16}{3}\varepsilon$.

If $g(x_1=2)=1$, then, since $\alpha_2\leq\alpha\leq\frac{2}{3}+\varepsilon$, $||\un_F-g||_2^2 \geq \frac{1}{9} - \frac{\varepsilon}{3}$. If $g(x_1=2)=0$, then, since $\alpha_2\geq\frac{2}{3}-\frac{16}{3}\varepsilon$, $||\un_F-g||_2^2 \geq \frac{2}{9} - \frac{16\varepsilon}{9}$. Then for $\varepsilon \leq \frac{1}{13}$ it always holds that $||\un_F-g||_2^2 \geq \frac{1}{9} - \frac{\varepsilon}{3} \geq \frac{1}{10}$.

Overall, 

$$\frac{1}{10}\leq ||\un_F-g||_2^2\leq 9K\varepsilon,$$
and so
$$\frac{d}{2^n} = \varepsilon \geq \frac{1}{90K}.$$
\end{proof}

\subsection{Examples and the difference between maximum and average degree}

\subsubsection{Examples and an upper bound}\label{ssec:ex-tensor}

In this subsection, we give four vertex sets $F\subseteq V(K_3^{\otimes n})$ for which $\Delta(K_3^{\otimes n}[F])=a_{K_3^{\otimes n}}(|F|)$, and hence the induced subgraph is optimal with respect to maximum degree. These examples also yield a simple upper bound on $\Delta_{K_3^{\otimes n}}(m)$ for every $m$. Figure~\ref{fig2} shows the resulting asymptotic bounds: the blue curve is the average-degree lower bound, and the orange curve is the upper bound. The curves coincide at $4/9$, $16/27$, $2/3$, and $8/9$, as shown in Claim~\ref{cl:examples-tensor}. The jumps in the orange curve at $1/3$ and $2/3$ are unavoidable by Theorem~\ref{thm-jump}; moreover, Section~\ref{ssec:no-big-perfect} shows that the gap between the two curves is unavoidable for almost every $\alpha>2/3$.

\begin{Claim}\label{cl:examples-tensor}

  \begin{enumerate}
  \item $\Delta_{K_3^{\otimes n}}(\frac{4}{9}\cdot 3^{n}) = a_{K_3^{\otimes n}}(\frac{4}{9}\cdot 3^{n})= \frac{1}{4}\cdot 2^{n}$;
  \item $\Delta_{K_3^{\otimes n}}(\frac{16}{27}\cdot 3^{n}) = a_{K_3^{\otimes n}}(\frac{16}{27}\cdot 3^{n}) = \frac{7}{16}\cdot 2^{n}$;
  \item $\Delta_{K_3^{\otimes n}}(\frac23\cdot 3^{n}) = a_{K_3^{\otimes n}}(\frac{2}{3}\cdot 3^{n})= \frac12\cdot 2^{n}$;
  \item $\Delta_{K_3^{\otimes n}}(\frac{8}{9}\cdot 3^{n}) = a_{K_3^{\otimes n}}(\frac{8}{9}\cdot 3^{n}) = \frac78\cdot 2^{n}$.
  \end{enumerate}
\end{Claim}

\begin{proof}
  In each case, the average-degree bound gives the lower bound. The corresponding constructions giving the upper bound are:
  \begin{enumerate}
  \item $F^1_n = \{0,1\}^2 \times \Z_3^{n-2}$;
  \item $F_n^2 = F_{n-1}^4 \times \{0,1\}$;
  \item $F^3_n = \{0,1\} \times \Z_3^{n-1}$;
  \item $F_n^4 = (\Z_3^4\backslash\{0000, 1110, 2220, 1201, 2011, 0121, 2102, 0212, 1022\}) \times \Z_3^{n-4}$.
  \end{enumerate}
\end{proof}

These constructions immediately give an upper bound for every $m$: $\Delta_{K_3^{\otimes n}}(m)\leq a_{K_3^{\otimes n}}(m_0)$, where $m_0$ is the smallest number among the four sizes in the claim and $3^n$ that satisfies $m\leq m_0$. 
A simple probabilistic thinning argument usually gives a sharper bound.

\begin{Claim}\label{cl:easy-prob-constr}
  Suppose that $\Delta_{K_3^{\otimes n}}(m)=d$ with $d\geq2$. Then, for every $p\in(0,1)$,
  $$\Delta_{K_3^{\otimes n}}\!\left(\left\lceil pm\left(1-\frac{1}{d^2}\right)\right\rceil\right) \leq pd + \sqrt{d\ln d}.$$
\end{Claim}
\begin{proof}
  Let $G$ denote the ambient product graph, and choose $F\subseteq V(G)$ with $|F|=m$ and $\Delta(G[F])=d$. Let $F_p$ be a $p$-random subset of $F$, and set
  $$B:=\{x\in F_p:\deg_{F_p}(x)>pd+\sqrt{d\ln d}\}.$$
  Conditional on $x\in F_p$, the random variable $\deg_{F_p}(x)$ is a sum of $\deg_F(x)$ independent Bernoulli variables. Hence, for $t>0$ and $\deg_F(x)>0$, the Chernoff--Hoeffding bound gives
  $$\Pp[x\in B]\leq p\,\Pp[\deg_{F_p}(x)\geq p\deg_F(x)+t\mid x\in F_p]
  \leq p\exp\!\left(-\frac{2t^2}{\deg_F(x)}\right)
  \leq p\exp\!\left(-\frac{2t^2}{d}\right).$$
  (The assertion is trivial when $\deg_F(x)=0$.) Taking $t=\sqrt{d\ln d}$ and summing over $x\in F$ yields $\mathbb{E}|B|\leq mp/d^2$. Therefore
  $$\mathbb{E}(|F_p|-|B|)\geq pm\left(1-\frac{1}{d^2}\right).$$
  For some realization, deleting $B$ leaves at least $\left\lceil pm(1-d^{-2})\right\rceil$ vertices and maximum degree at most $pd+\sqrt{d\ln d}$. Taking a subset of the required size proves the claim.
\end{proof}

The same argument applies to any graph with a prescribed maximum degree, in particular for induced subgraphs of $K_3^{\Box n}$, and we shall use the analogue in Claim~\ref{cl:easy-pb-constr-Hamming}. 

For constant $p$ and $d=\Omega(2^n)$, we obtain
$$\Delta_{K_3^{\otimes n}}\bigl(pm(1-O(2^{-2n}))\bigr)\leq pd\bigl(1+O(n2^{-n/2})\bigr).$$ This yields Figure~\ref{fig2}.

\begin{figure}[h!]
\includegraphics[width=\textwidth]{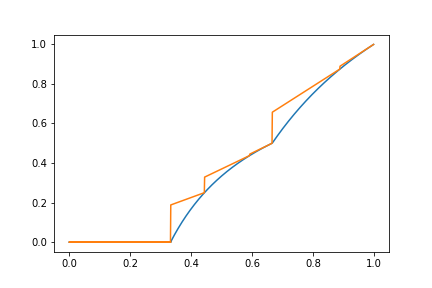}
\caption{Upper and lower bounds on $\inf_{n,\,m\geq \alpha 3^n} \frac{\Delta_{K_3^{\otimes n}}(m)}{2^n}$ as a function of $\alpha$.}
\label{fig2}

\end{figure}

We shall also use the following elementary observation, whose Hamming-product analogue is slightly less direct, see Subsection~\ref{sec:transform23hamming}.

\begin{Claim}\label{cl:transform23tensor}
For any $m$ we have $\Delta_{K^{\otimes (n+1)}_3}(2m) \leq \Delta_{K^{\otimes n}_3}(m)$, since for any $F$ realizing $\Delta_{K^{\otimes n}_3}(m)$ we have $\Delta(K_3^{\otimes (n+1)}[F\times \{0,1\}]) = \Delta_{K^{\otimes n}_3}(m)$.
\end{Claim}

\subsubsection{The gap between the minimum possible average and maximum degrees (for large $|F|$)}\label{ssec:no-big-perfect}

In this subsection, we prove that for almost every $\alpha:=m/3^n\geq2/3$, the minimum possible maximum degree of an induced subgraph of the tensor product on $m$ vertices exceeds the minimum possible average degree by a positive fraction of the degree of the whole graph. This fraction is independent of $n$, although it depends on $\alpha$. We also show that the two minima can coincide, for some $n$, at only two densities in the range $2/3\leq\alpha<1$.

Let $G$ be a $\Delta$-regular graph on $N$ vertices. Let $F\subseteq V(G)$ satisfy $|F|=\alpha N$ and $\Delta(G[F])\leq d=:\beta\Delta$. Let $A$ be the adjacency matrix of $G$, and write its eigenvalues as $\Delta=\lambda_1\geq\lambda_2\geq\cdots\geq\lambda_N$. Decompose the characteristic vector of $F$ as $\un_F=f_1+\cdots+f_N$, where $f_i$ lies in the $\lambda_i$-eigenspace. Finally, set
$$\delta:=\frac{1}{\Delta}\min_{2\leq i\leq N}|\lambda_i-(-\Delta+d)|.$$

$\Delta(G[F]) \leq d$ implies

$$\left\|\Big(\Big(1-\frac{d}{\Delta}\Big)I_N+\frac{1}{\Delta}A\Big)\un_F\right\|_{\infty}\leq1,$$

because

$$\Big(\Big(\Big(1-\frac{d}{\Delta}\Big)I_N + \frac{1}{\Delta}A\Big)\un_F\Big)_v = \Big(1-\frac{d}{\Delta}\Big)\mathbf{1}_{\{v\in F\}} + \frac{1}{\Delta}\deg_F(v) = \begin{cases}
  \frac{\deg_F(v)}{\Delta}\leq 1, & v\notin F;
\\ 1 - \frac{d - \deg_F(v)}{\Delta} \leq 1, & v\in F.\end{cases}$$

Thus

$$\bigg\|\Big(\Big(1-\frac{d}{\Delta}\Big)I_N+\frac{1}{\Delta}A\Big)\un_F\bigg\|_2^2\leq N.$$

Hence,

\begin{multline*}
  N \geq (1-\beta)^2|F| + \frac{2}{\Delta}(1-\beta)\un_F^TA\un_F + \frac{1}{\Delta^2}\un_F^TA^2\un_F = (1-\beta)^2|F| + \sum_{i=1}^N||f_i||_2^2\Big(\frac{2}{\Delta}(1-\beta)\lambda_i + \frac{1}{\Delta^2}\lambda_i^2\Big) =\\= (1-\beta)^2|F| + \frac{|F|^2}{N}(2(1-\beta)+1) + \sum_{i=2}^N||f_i||_2^2\Big(\Big(1-\beta+\frac{\lambda_i}{\Delta}\Big)^2-(1-\beta)^2\Big) \geq \\ \geq (1-\beta)^2|F| + \frac{|F|^2}{N}(2(1-\beta)+1) + \Big(|F|-\frac{|F|^2}{N}\Big)(\delta^2-(1-\beta)^2).
\end{multline*}

Dividing by $N$, we get

$$1 \geq (1-\beta)^2\alpha + \alpha^2(3- 2\beta) + (\alpha - \alpha^2)(\delta^2 - (1-\beta)^2) = \alpha^2(2-\beta)^2 + (\alpha - \alpha^2)\delta^2.$$

For $K_3^{\otimes n}$ and $\alpha\geq2/3$, the normalized average-degree bound is
$$\beta_{\mathrm{av}}(\alpha):=2-\frac{1}{\alpha},$$
and hence $\alpha(2-\beta_{\mathrm{av}}(\alpha))=1$.

Consequently, if $2/3\leq\alpha<1$ and
$$\alpha\notin\left\{\frac{1}{1+2^{-(2r+1)}}:r\in\Z_{\geq0}\right\},$$
then there exists $\varepsilon(\alpha)>0$ such that, for every sequence $\rho_n\to0$,
$$
\liminf_{n\to\infty}
\min_{\substack{F\subseteq V(K_3^{\otimes n})\\ |F|\geq(\alpha-\rho_n)3^n}}
\frac{\Delta(K_3^{\otimes n}[F])}{2^n}
\geq\beta_{\mathrm{av}}(\alpha)+\varepsilon(\alpha).
$$
Thus there are no asymptotically optimal examples of the type constructed for the Hamming product in Section~\ref{sec:hamming}. The proof also shows that, at a density where the average-degree bound can be attained by the maximum degree, the characteristic function must have the form $f_1+f_j$, where $\lambda_j=-2^n+d$. Conversely, any such Boolean function of the required size gives an example.

At least two such examples exist: 

\begin{itemize}
  \item trivial example: $\beta = \frac{1}{2}$, $\alpha = \frac{2}{3}$, $\un_F=\mathbf{1}_{\{x_1\neq0\}}= \frac{2}{3}-\frac{1}{3}\omega^{x_1} - \frac{1}{3}\omega^{-x_1}$, the third example in Claim~\ref{cl:examples-tensor};
  \item another example: $\beta = \frac{7}{8}$, $\alpha = \frac{8}{9}$,
  \begin{multline*}
  \un_F = \frac{8}{9}-\frac{1}{9}(\omega^{x_1 + x_2 + x_3} +\omega^{-x_1-x_2-x_3}
  +\omega^{x_2-x_3+x_4}
  +\omega^{-x_2+x_3-x_4}
  +\\+\omega^{x_1-x_2+x_4}
  +\omega^{-x_1+x_2-x_4}
  +\omega^{-x_1+x_3+x_4}
  +\omega^{x_1-x_3-x_4})
  \end{multline*}
  or, equivalently,
  $F = (\Z_3^4\backslash\{0000, 1110, 2220, 1201, 2011, 0121, 2102, 0212, 1022\}) \times \Z_3^{n-4}$, the fourth example in Claim~\ref{cl:examples-tensor};
\end{itemize}

The preceding argument also shows that, in contrast to the Hamming-product case (Statement~\ref{st-perfect-ex}), there are no "perfect" examples for size fractions $1 > \alpha \geq \frac{2}{3}$ apart from the two values above:

\begin{Claim}
  Suppose that for a subset $F \subset \Z_3^n$ with $3^n > |F| \geq 2\cdot 3^{n-1}$ we have $\Delta(K_3^{\otimes n}[F])=a_{K_3^{\otimes n}}(|F|)$. Then we have $\frac{|F|}{3^n} = \frac{2}{3}$ or $\frac{|F|}{3^n} = \frac{8}{9}$.
\end{Claim}

\begin{proof}
  By the above, we should have $|F| = \frac{1}{1+2^{-(2k+1)}} \cdot 3^n = \frac{2^{2k+1}\cdot 3^n}{2^{2k+1}+1}$. Since $|F| \in \Z_{> 0}$, this implies $2^{2k+1}+1 = 3^q$ for some non-negative integer $q$. It suffices to show that $(1,1)$ and $(3, 2)$ are the only non-negative integer solutions $(p,q)$ for $2^{p}+1 = 3^q$. Suppose that $p \geq 4$. Then $3^q \equiv 1 \mod 16$. Then, since $\text{ord}_{\Z_{16}^*}(3) = 4$, we have $q = 4q'$, $q'\in \Z_{\geq 0}$. Then $2^p \equiv 3^{4q'} - 1 \equiv 1^{q'} - 1 = 0 \mod 5$, contradiction. Also $2 \neq 3^q$, $5 \neq 3^q$, and so there are no solutions with $p = 0$ or $p = 2$.
\end{proof}

\subsubsection{No perfect examples for small degrees}\label{ssec:no-small-perfect}

\begin{Statement}
  There is no $3^{n-1} < m \leq \frac{6}{17} \cdot 3^n$ such that $\Delta_{K_3^{\otimes n}}(m) = a_{K_3^{\otimes n}}(m)$. 
\end{Statement}

\begin{proof}
  Assume the contrary, i.e., there exists $F$ of size $m$ such that $\Delta_{K_3^{\otimes n}}(m) = a_{K_3^{\otimes n}}(m)$. Set $d:=\Delta_{K_3^{\otimes n}}(m)$. 

  Denote $\alpha := \frac{|F|}{3^n}$. Let, as in the previous sections, $F_i := F\cap \{x_1 = i\}$ for $i = 0, 1, 2$. Let $\alpha_i := \frac{|F_i|}{3^{n-1}}$. Assume without loss of generality $\alpha_0 \geq \alpha_1 \geq \alpha_2$. Since $m > 3^{n-1}$, $\alpha_1 > 0$. Delete some $(\alpha - \frac{1}{3})\cdot 3^n$ elements from $F$ getting a set $F'$ such that $\alpha'_0 + \alpha'_1 + \alpha'_2 = 1$ and still $\alpha'_0 \geq \alpha'_1 \geq \alpha'_2$ and $\alpha'_1 > 0$.

  By exactly the same argument as in the proof of part~\ref{thm:jump0} of Theorem~\ref{thm-jump}, we have $\alpha'_0 - \alpha'_2 \leq 8\varepsilon - 8\varepsilon^2$.

  Since we deleted $(3\alpha -1)\cdot 3^{n-1} = 3\alpha\varepsilon\cdot 3^{n-1}$ (since $d$ is the minimum possible average degree), for the actual set $F$ we get $\alpha_0 - \alpha_2 \leq 8\varepsilon - 8\varepsilon^2 + 3\alpha\varepsilon \leq 9\varepsilon$ for $\varepsilon \leq \frac{1}{2}$.

  This gives the following upper bound on the absolute values of the first level Fourier coefficients $\hat f(\pm e_1)$ of the characteristic function $\un_F =: f$ of $F$.

  $$\hat f(e_1) = \frac{1}{3^n}\sum_{x\in \Z_3^n} f(x)\omega^{-x_1} = \frac{1}{3^n}(|F_0|+\omega^2|F_1| + \omega|F_2|) = \frac{\alpha_0 + \omega^2\alpha_1 + \omega\alpha_2}{3};$$

  \begin{multline*}
  |\hat f(e_1)|^2 = |\hat f(-e_1)|^2 = \frac{\alpha_0^2 + \alpha_1^2 + \alpha_2^2 - \alpha_0\alpha_1 - \alpha_1\alpha_2 - \alpha_0\alpha_2}{9} =\\= \frac{(\alpha_0 - \alpha_1)^2+(\alpha_1-\alpha_2)^2+(\alpha_0-\alpha_2)^2}{18} \leq \frac{(\alpha_0 - \alpha_2)^2}{9} \leq 9\varepsilon^2.
  \end{multline*}

  Recall that for a family of maximum degree $d$ the Fourier decomposition is concentrated on the first two levels, namely, $||P_{\geq 2}(f)||_2^2 \leq \frac{2}{3}(\alpha\frac{d}{2^{n-2}} - 2(3\alpha^2 - \alpha)) = \frac{2}{3}(4\alpha\varepsilon-6\alpha^2\varepsilon) \leq \frac{4}{9}\varepsilon$.

  Note that the characteristic function of the set $F+e_1$ is $f^+(x) = f(x-e_1)$, and thus

  \begin{multline*}
  \gamma:=\frac{|F\mathbin{\triangle}(F+e_1)|}{3^n} = ||f - f^+||_2^2 = \sum_{y\in \Z_3^n}|\hat f(y)|^2|1-\omega^{-y_1}|^2 = 3\sum_{y : y_1 \neq 0}|\hat f(y)|^2\leq \\ \leq 3(|\hat f(e_1)|^2 + |\hat f(-e_1)|^2 + ||P_{\geq 2}(f)||_2^2) \leq 3(18\varepsilon^2 + \frac{4}{9}\varepsilon) = 54\varepsilon^2 + \frac{4}{3}\varepsilon.
  \end{multline*}

  Let $L_t^\psi$ denote the number of lines in $\Z_3^n$ parallel to $\psi\in\{1,2\}^n$ that contain exactly $t$ elements of $F$, and set $L_t:=\sum_{\psi\in\{1,2\}^n}L_t^\psi$. Then, for every $\psi$,

  $$|F| = L^\psi_1 + 2L_2^\psi + 3 L_3^\psi,$$
  $$3^{n-1} = L_0^\psi + L^\psi_1 + L_2^\psi + L_3^\psi,$$
  and 

  \begin{multline*}
  2|E(K_3^{\otimes n}[F])| = \sum_{\psi\in \{1,2\}^n}(L_2^\psi + 3L_3^\psi) = \sum_{\psi\in \{1,2\}^n}(|F| - 3^{n-1} + L_0^\psi + L_3^\psi) =\\= 2^{n}(|F| - 3^{n-1}) + L_0 + L_3 = |F|a_{K_3^\otimes n}(|F|) + L_0 + L_3.
  \end{multline*}

  Hence, if $F$ has maximum degree $a_{K_3^{\otimes n}}(|F|)$, then in particular $L_0 = 0$, i.e., there are no empty lines, meaning that for any $x\in \Z_3^n$ and $\psi\in\{1,2\}^n$ we have $\{x, x+\psi, x+2\psi\} \cap F \neq \emptyset$.
  
  Since $x\mapsto x+e_1$ is an automorphism of $K_3^{\otimes n}$, the translate $F+e_1$ has the same property and maximum degree $d$. Consider $x\in(F+e_1)\setminus F$. Every line through $x$ parallel to a vector in $\{1,2\}^n$ contains a point of $F$, so $\deg_F(x)\geq2^{n-1}$. Since $F+e_1$ has maximum degree $d$, at least $2^{n-1}-d$ of these neighbors lie in $F\setminus(F+e_1)$. Interchanging $F$ and $F+e_1$ gives the same conclusion on the other side. Hence the subgraph induced by $F\mathbin{\triangle}(F+e_1)$ has minimum degree at least $2^{n-1}-d$. Therefore 

  $$(2^{n-1} - d)\gamma \leq \frac{\un_{F\mathbin{\triangle}(F+e_1)}^TB_n\un_{F\mathbin{\triangle}(F+e_1)}}{3^n} \leq 2^n\gamma^2 + 2^{n-2}(\gamma-\gamma^2).$$

  If $\gamma > 0$, this implies $\gamma \geq \frac{1}{3} - \frac{4}{3}\varepsilon$, hence $54\varepsilon^2 + \frac{4}{3}\varepsilon\geq \frac{1}{3} - \frac{4}{3}\varepsilon$ and thus $\varepsilon > \frac{1}{18}$. 

  Thus, if $\varepsilon \leq \frac{1}{18}$, then $\gamma = 0$ and so $F = F+e_1$ and $f(x) = f(x-e_1)$ for all $x\in \Z_3^n$, and by exactly the same argument this holds for any $e_i$ instead of $e_1$. But the only such $f$'s are constants, contradiction. 

  $\varepsilon > \frac{1}{18}$ implies (for families with minimum possible average degree $d$) that $\alpha > \frac{6}{17}$.
  
\end{proof}

\section{Upper bounds on maximum degree for the Hamming product}\label{sec:hamming}

In this section, we present families, parametrized by $n$, of vertex subsets of $K_3^{\Box n}$ with corresponding induced subgraphs having small maximum degree. In particular, we prove Theorem~\ref{thm-ex-hamming}. We also give transformations that produce new subsets of reasonably small maximum degree from existing constructions and from large induced subgraphs of the hypercube with small maximum degree. 

Although the normalized minimum average degrees of the tensor and Hamming products coincide, their minimum possible maximum degrees behave quite differently. In particular, there is no jump at zero for the Hamming product, as noted in the introduction. As we show below, for many densities $\alpha=m/3^n$, the minimum possible maximum degree for the Hamming product is asymptotically close to, and sometimes equal to, the minimum possible average degree. This contrasts with the tensor-product behavior established in the previous section. Throughout this section, a term such as $o_\ell(1)$ tends to zero as $n\to\infty$ with $\ell$ fixed.

\begin{Definition}
  For any $\alpha\in[0,1]$, define $$\bar \Delta_{\Box, n}(\alpha) := \frac{\min_{F\subset V(K_3^{\Box n}) : |F| \geq \alpha\cdot 3^n}\Delta(K_3^{\Box n}[F])}{2n};$$ $$\bar \Delta_{\Box}(\alpha) := \inf_{n}\bar\Delta_{\Box, n}(\alpha);$$ $$\bar \Delta_{\Box}^*(\alpha) := \sup_{\alpha' <\alpha}\bar\Delta_{\Box}(\alpha').$$ 
  
  Denote $$\bar a_\Box(\alpha) = \begin{cases}
  0,& \alpha\leq\frac{1}{3}\\
  1 - \frac{1}{3\alpha},& \frac{1}{3} \leq \alpha \leq 
  \frac{2}{3}\\
  2 - \frac{1}{\alpha}, & \alpha \geq \frac{2}{3}.
  \end{cases}$$
  Section~\ref{sec-average} shows that, for every $n\in\N$,
  $$\min_{\substack{F\subseteq V(K_3^{\Box n})\\ |F|\geq \alpha3^n}}\frac{a(K_3^{\Box n}[F])}{2n}
  =\bar a_\Box\!\left(\frac{\lceil\alpha3^n\rceil}{3^n}\right).$$
  Consequently,
  $$\bar a_\Box(\alpha)=\inf_{n\geq1}\min_{\substack{F\subseteq V(K_3^{\Box n})\\ |F|\geq \alpha3^n}}\frac{a(K_3^{\Box n}[F])}{2n}.$$

  Clearly, $\bar a_\Box(\alpha) \leq \bar\Delta^*_\Box(\alpha) \leq \bar\Delta_\Box(\alpha)$.
\end{Definition}

In this section, in particular, we obtain an upper bound on $\bar\Delta_\Box(\alpha)$, depicted in orange in Figure~\ref{fig:hamming-bounds}. The bound follows from the constructions in Theorem~\ref{thm-ex-hamming} and the elementary claim below. The blue line is the average degree lower bound.

\begin{Claim}\label{cl:easy-pb-constr-Hamming}
  Suppose that $\Delta_{K_3^{\Box n}}(m)=d$ with $d\geq2$. Then, for every $p\in(0,1)$,
  $$\Delta_{K_3^{\Box n}}\!\left(\left\lceil pm\left(1-\frac{1}{d^2}\right)\right\rceil\right)\leq pd+\sqrt{d\ln d}.$$
\end{Claim}

\begin{figure}[h!]
\includegraphics[width=\textwidth]{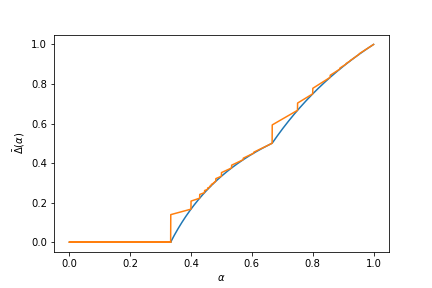}
\caption{Upper and lower bounds on $\inf_{n,\,m\geq \alpha 3^n} \frac{\Delta_{K_3^{\Box n}}(m)}{2n}$ as a function of $\alpha$.}
\label{fig:hamming-bounds}
\end{figure}

\subsection{Three simple asymptotically optimal examples}

\begin{Statement}
  $\bar\Delta_\Box^*(\frac{2}{5}) = \bar a_\Box(\frac{2}{5}) = \frac{1}{6}$, $\bar\Delta_\Box^*(\frac{1}{2}) = \bar a_\Box(\frac{1}{2}) = \frac{1}{3}$ and $\bar\Delta_\Box^*(\frac{3}{4}) = \bar a_\Box(\frac{3}{4}) = \frac{2}{3}$.
\end{Statement}

\begin{proof}
  It is enough to find for each $\alpha = \frac{2}{5}, \frac{1}{2}, \frac{3}{4}$ a sequence of integers $n_i \to \infty$ and a vertex subset $F_{n_i} \subset \Z_3^{n_i}$ of size $m_{n_i} \geq (1-o(1))\alpha\cdot 3^{n_i}$ such that $\Delta(K_3^{\Box n_i}[F_{n_i}]) \leq (1+o(1)) \bar a_{\Box}(\alpha) \cdot 2n_i$. To simplify notation, we henceforth write $n$ in place of $n_i$. 

\vspace{3mm}

  \textbf{An example for $m \geq (\frac{1}{2} - o(1))\cdot 3^n$ and $\Delta \leq (\frac{1}{3}+o(1))\cdot 2n$:}
  
  The Chernoff-Hoeffding bound implies that for uniformly random $x \in \{0,1,2\}^n$ for each $\ell = 0, 1, 2$ we have

$$\Pp[\#\{i \in [n] : x_i = \ell\} > \frac{n}{3} + n^{\frac{3}{4}}] \leq e^{-2\sqrt{n}}.$$

Note that $\sum_{i=1}^n x_i\equiv\#\{i\in[n]:x_i=1\}\pmod 2$. Thus $\#\{x : \sum_{i=1}^n x_i \equiv 0\,\, mod\,\,2\} - \#\{x : \sum_{i=1}^n x_i \equiv 1\,\, mod\,\,2\} = \sum_{\ell=0}^n(-1)^\ell2^{n-\ell}\binom{n}{\ell}=(2-1)^n=1$.

Let $$F := \{x : \#\{i \in [n] : x_i = \ell\} \leq \frac{n}{3} + n^{\frac{3}{4}}, \ell = 0,2; \sum_{i=1}^n x_i \equiv 0 \mod 2\}.$$
By the preceding estimates, $|F| \geq (\frac{1}{2} - 3e^{-2\sqrt{n}})\cdot 3^n$ for $n \geq 4$.

On the other hand, to change one coordinate but not the parity of the sum of coordinates we need to either replace 0 by 2 or vice versa. Thus $\Delta(K_3^{\Box n}[F]) \leq \frac{2}{3}n + 2n^{\frac{3}{4}} = (\frac{1}{3} + n^{-\frac{1}{4}})\cdot 2n$.

Note that $\bar a_{\Box}(\frac{1}{2}) = 1 - \frac{1}{3/2} = \frac{1}{3}$.

\vspace{3mm}

\textbf{An example for $m \geq (\frac{3}{4} - o(1))\cdot 3^n$ and $\Delta \leq (\frac{2}{3}+o(1))\cdot 2n$:}

Clearly, there exists $\ell \in \{0,1,2,3\}$, such that $\#\{x : \sum_{i=1}^n x_i \equiv \ell \mod 4\} \leq \frac{3^n}{4}$. Let $$F := \{x : \#\{i \in [n] : x_i = t\} \leq \frac{n}{3} + n^{\frac{3}{4}}, t = 0, 1, 2; \sum_{i=1}^n x_i \not\equiv \ell \mod 4\}.$$ Then $|F| \geq (\frac{3}{4} - 3e^{-2\sqrt{n}})\cdot 3^n$. On the other hand, if the sum of coordinates of $x$ is $\ell+2 \mod 4$, then if we replace 0 by 2 or 2 by 0, we will get residue $\ell$. If the sum is $\ell+1 \mod 4$, we will get sum $\ell$ by changing 1 to 0 and 2 to 1. Finally, if the sum is $\ell+3 \mod 4$, we will get sum $\ell$ by changing 0 to 1 and 1 to 2. Hence, $\Delta(K_3^{\Box n}[F]) \leq \frac{4}{3}n + 4n^{\frac{3}{4}} = (\frac{2}{3} + 2n^{-\frac{1}{4}})\cdot 2n$.

Note that $\bar a_{\Box}(\frac{3}{4}) = 2 - \frac{1}{3/4} = \frac{2}{3}$.

\vspace{3mm}

\textbf{An example for $m \geq (\frac{2}{5} - o(1))\cdot 3^n$ and $\Delta \leq (\frac{1}{6}+o(1))\cdot 2n$:}

Since $3^n = \sum_{\ell = 0}^4 \#\{x : \sum x_i \equiv \ell \mod 5\} = \frac{1}{2}\sum_{\ell = 0}^4 (\#\{x : \sum x_i \equiv \ell \mod 5\} + \#\{x : \sum x_i \equiv \ell+2 \mod 5\})$, there exists an $\ell$ such that $\#\{x : \sum x_i \equiv \ell \mod 5\} + \#\{x : \sum x_i \equiv \ell+2 \mod 5\} \geq \frac{2}{5}\cdot 3^n$. Let $$F := \{x : \#\{i \in [n] : x_i = t\} \leq \frac{n}{3} + n^{\frac{3}{4}}, t = 0, 1, 2; \sum_{i=1}^n x_i \equiv \ell \,\, or \sum_{i=1}^n x_i \equiv \ell + 2 \mod 5\}.$$ Then $|F| \geq (\frac{2}{5} - 3 e^{-2\sqrt{n}}) \cdot 3^n$ and $\Delta(K_3^{\Box n}[F]) \leq \frac{n}{3} + n^{\frac{3}{4}} = (\frac{1}{6} + \frac{1}{2}n^{-\frac{1}{4}})\cdot 2n$.

Note that $\bar a_\Box(2/5)=1-1/(3\cdot 2/5)=1/6$.
  
\end{proof}

\subsection{Infinite sequences of asymptotically optimal examples}

\begin{Statement}\label{st-sequence1}
  For each integer $\ell \geq 1$ 

  $$\bar \Delta_\Box^*\bigg(1 - \frac{1}{2\ell+1}\bigg) = \bar a_\Box\bigg(1 - \frac{1}{2\ell+1}\bigg),$$
  
  that is, there is a sequence $m_n^{(\ell)} = (1 - \frac{1}{2\ell+1})\cdot 3^n \cdot (1 - o(1))$ (with $n\to \infty$) such that
  $$\Delta_{K_3^{\Box n}}(m_n^{(\ell)}) = \bar a_{\Box}\bigg(1 - \frac{1}{2\ell+1}\bigg)\cdot 2n(1 + o(1)).$$
\end{Statement}

\begin{proof}
  Suppose that $n$ is divisible by $2\ell$. Let $F_R\subseteq\{0,1,2\}^n$ be the set of sequences $x$ such that for each $t = 0,1,2$ and each $q = 1,\ldots,2\ell$ it holds that $\#\{i\in [\frac{(q-1)n}{2\ell}+1, \frac{qn}{2\ell}] : x_i = t\} = (1 + o_{\ell}(1))\frac{n}{6\ell}$ and $$\sum_{q=1}^{2\ell} \sum_{i = \frac{(q-1)n}{2\ell}+1}^{\frac{qn}{2\ell}} q\cdot x_i \not\equiv R \mod 2\ell + 1.$$ 

  By the Chernoff-Hoeffding bound and an averaging argument $|F_R| \geq (1-o_\ell(1))(1 -\frac{1}{2\ell+1})\cdot 3^n$ for at least one $R$. Here and below, the tolerance implicit in $o_\ell(1)$ is chosen to tend to zero slowly enough for the Chernoff--Hoeffding estimate to apply.

  Since $2$ is invertible modulo $2\ell + 1$, for each $0 \neq r \in \Z_{2\ell + 1}$ there are the following ways to change the residue of the weighted sum by $r$ by changing one coordinate of $x$:
  \begin{itemize}
  \item change $0$ to $1$ or $1$ to $2$ on the interval $[\frac{(q-1)n}{2\ell}+1, \frac{qn}{2\ell}]$ for the unique $q \equiv r \mod 2\ell + 1$;
  \item change $1$ to $0$ or $2$ to $1$ on the interval $[\frac{(q-1)n}{2\ell}+1, \frac{qn}{2\ell}]$ for the unique $q \equiv -r \mod 2\ell + 1$;
  \item change $0$ to $2$ on the interval $[\frac{(q-1)n}{2\ell}+1, \frac{qn}{2\ell}]$ for the unique $q$ such that $2q \equiv r \mod 2\ell + 1$;
  \item change $2$ to $0$ on the interval $[\frac{(q-1)n}{2\ell}+1, \frac{qn}{2\ell}]$ for the unique $q$ such that $-2q \equiv r \mod 2\ell + 1$.
  \end{itemize}
  Since exactly one non-zero change of the residue sends the residue of $x\in F_R$ to $R$, 
  $$\Delta(K_3^{\Box n}[F_R]) \leq 2n - (\frac{n}{3\ell} + \frac{n}{3\ell} + \frac{n}{6\ell} + \frac{n}{6\ell})(1 + o_\ell(1)) = 2n(1 - \frac{1}{2\ell})(1 + o_{\ell}(1)).$$
  On the other hand, $$\bar a_{\Box}(1 - \frac{1}{2\ell+1}) = 2 - \frac{1}{1 - \frac{1}{2\ell+1}} = 1 - \frac{1}{2\ell}.$$
\end{proof}

\begin{Statement}
  For each integer $\ell \geq 1$ 
  $$\bar \Delta^*_\Box\bigg(\frac{2(2^\ell - 1)}{4(2^\ell-1) + 1}\bigg) = \bar a_\Box\bigg(\frac{2(2^\ell - 1)}{4(2^\ell-1) + 1}\bigg)$$
  that is, there is a sequence $m_n^{(\ell)} = \frac{2(2^\ell - 1)}{4(2^\ell-1) + 1}\cdot 3^n \cdot (1 - o_\ell(1))$ (with $n\to \infty$) such that
  $$\Delta_{K_3^{\Box n}}(m_n^{(\ell)}) = \bar a_\Box\bigg(\frac{2(2^\ell - 1)}{4(2^\ell-1) + 1}\bigg)\cdot 2n(1 + o_\ell(1)).$$
\end{Statement}

\begin{proof}
  Set $s:=2^\ell-1$ and suppose that $n$ is divisible by $s$. For a residue $R$ modulo $4s+1$, let $F_R\subseteq\{0,1,2\}^n$ be the set of sequences such that $$\sum_{q = 0}^{\ell-1} (2^{\ell - q}-1)\sum_{i=\frac{(2^{q}-1)n}{2^\ell-1}+1}^{\frac{(2^{q+1}-1)n}{2^\ell - 1}}x_i \in \{R+1,R+3,\ldots,R+(4s-1)\}$$ and for each $t = 0,1,2$ and each $q = 0,\ldots,\ell-1$ it holds that $$|\{i \in [\frac{(2^{q}-1)n}{2^\ell-1}+1, \frac{(2^{q+1}-1)n}{2^\ell - 1}] : x_i = t\}| = (1+o_\ell(1))\frac{2^q n}{3(2^\ell - 1)}.$$ By the Chernoff-Hoeffding bound and an averaging argument, $|F_R|\geq (1-o_\ell(1))\frac{2s}{4s+1}\cdot 3^{n}$ for at least one residue $R$. Let us call $[\frac{(2^{q}-1)n}{2^\ell-1}+1, \frac{(2^{q+1}-1)n}{2^\ell - 1}]$ the $q$-th interval. For each $p = 2,3,\ldots,\ell$ the residue of the weighted sum can increase (decrease) by $2^{p}-2$ by increasing (decreasing) one coordinate in ($\ell - p + 1$)-th interval by $2$ (each in $\frac{2^{\ell - p + 1}n}{3(2^\ell -1)}(1+o_\ell(1))$ ways), and it can increase (decrease) by $2^{p}-1$ by increasing (decreasing) one coordinate in the ($\ell-p$)-th interval by $1$ (each in $2\cdot \frac{2^{\ell - p}n}{3(2^\ell -1)}(1+o_\ell(1)) = \frac{2^{\ell - p+1}n}{3(2^\ell -1)}(1+o_\ell(1))$ ways). Also, it can increase (decrease) by $1$ by increasing (decreasing) one coordinate in the $(\ell-1)$-th interval by $1$ (in $2\cdot\frac{2^{\ell-1}n}{3(2^\ell-1)}(1+o_\ell(1))$ ways) and increase (decrease) by $2^{\ell+1}-2 = 2s$ by increasing (decreasing) one coordinate in the $0$-th interval by $2$ (in $\frac{n}{3(2^\ell-1)}(1+o_\ell(1))$ ways).

  For each $r\in\mathcal R:=\{R+1,R+3,\ldots,R+(4s-1)\}$ and each $p=2,3,\ldots,\ell$, exactly one of $r+2^p-2$ and $r+2^p-1$ lies in $\mathcal R$, and exactly one of $r-(2^p-2)$ and $r-(2^p-1)$ lies in $\mathcal R$. In addition, exactly one of $r+(2^{\ell+1}-2)$ and $r-(2^{\ell+1}-2)$ lies in $\mathcal R$, whereas $r\pm1$ never does. Thus every vertex of $K_3^{\Box n}[F_R]$ has degree at most
  $$2\cdot \sum_{p=2}^{\ell}\frac{2^{\ell - p+1}n}{3(2^\ell -1)}(1+o_\ell(1)) + \frac{n}{3(2^{\ell}-1)}(1+o_\ell(1)) = \frac{2^{\ell+1}-3}{6(2^\ell - 1)}\cdot 2n(1+o_\ell(1)).$$
  On the other hand, since $\frac{2(2^\ell - 1)}{4(2^\ell-1) + 1} < \frac{1}{2} <\frac{2}{3}$,
  $$\bar a_\Box(\frac{2(2^\ell - 1)}{4(2^\ell-1) + 1}) = 1 - \frac{4(2^\ell-1)+1}{6(2^\ell - 1)} = \frac{2^{\ell+1}-3}{6(2^\ell - 1)}.$$
\end{proof}

\begin{Statement}
  For each integer $\ell \geq 1$ 

  $$\bar \Delta^*_\Box\bigg(\frac{2\ell + 1}{4\ell + 3}\bigg) = \bar a_\Box\bigg(\frac{2\ell + 1}{4\ell + 3}\bigg),$$
  
  that is, there is a sequence $m_n^{(\ell)} = \frac{2\ell + 1}{4\ell + 3}\cdot 3^n \cdot (1 - o_\ell(1))$ such that
  $$\Delta_{K_3^{\Box n}}(m_n^{(\ell)}) = \bar a_{\Box}\bigg(\frac{2\ell+1}{4\ell + 3}\bigg)\cdot 2n (1 + o_\ell(1)).$$
\end{Statement}

\begin{proof}
  For each $t=1,2,\ldots,\ell$, define $\phi_t: \{0,1,2\} \to \Z_{4\ell+3}$: $\phi_t(0) = 0$, $\phi_t(1) = 1$, $\phi_t(2) = 2t$. 
  
  Suppose $n$ is divisible by $2\ell + 1$. Let $F_R \subset \{0,1,2\}^n$ be the family of sequences $x$ such that $$\sum_{t = 1}^{\ell}\sum_{i=\frac{2(t-1)n}{2\ell + 1}+1}^{\frac{2t n}{2\ell + 1}}\phi_t(x_i) + \sum_{i = \frac{2\ell n}{2\ell + 1}+1}^n (2\ell + 1)\cdot x_i \mod 4\ell + 3 \in \{R+1,R+3,\ldots,R+4\ell+1\}$$ and $$\#\{i\in [\frac{2(t-1)n}{2\ell + 1}+1, \frac{2t n}{2\ell + 1}]: x_i = a\} = (1+o_\ell(1))\frac{2n}{3(2\ell+1)}$$ for each $t\in \{1,2,\ldots,\ell\}$ and $a \in \{0,1,2\}$ and also $$\#\{i\in [\frac{2\ell n}{2\ell + 1}+1, n]: x_i = a\} = (1+o_\ell(1))\frac{n}{3(2\ell+1)}.$$ By an averaging argument, and the Chernoff-Hoeffding bound for at least one residue $R$ we have $|F_R| \geq \frac{2\ell + 1}{4\ell + 3}\cdot 3^n \cdot (1 - o_\ell(1))$. 

  Let us call $[\frac{2(t-1)n}{2\ell + 1}+1, \frac{2t n}{2\ell + 1}]$ the $t$-th interval and $[\frac{2\ell n}{2\ell + 1}+1, n]$ the last interval. 

  The possible changes $r$ of the residue of the sum by changing one coordinate are the following:
  \begin{itemize}
  \item $r = 2t$ for $t = 1,\ldots,\ell$: $0 \to 2$ on the $t$-th interval, $\frac{2n}{3(2\ell +1)}(1 + o_\ell(1))$ ways;
  \item $r = 2t + 1$ for $t = 1,\ldots,\ell-1$: $1 \to 2$ on the $(t + 1)$-th interval, $\frac{2n}{3(2\ell +1)}(1 + o_\ell(1))$ ways;
  \item $r = 2\ell +1$: $0 \to 1$ or $1 \to 2$ on the last interval, $\frac{2n}{3(2\ell +1)}(1 + o_\ell(1))$ ways;
  \item $r = -2t$ for $t = 1,\ldots,\ell$: $2 \to 0$ on the $t$-th interval, $\frac{2n}{3(2\ell +1)}(1 + o_\ell(1))$ ways;
  \item $r = -(2t + 1)$ for $t = 1,\ldots,\ell-1$: $2 \to 1$ on the $(t + 1)$-th interval, $\frac{2n}{3(2\ell +1)}(1 + o_\ell(1))$ ways;
  \item $r = -(2\ell +1)$: $1 \to 0$ or $2 \to 1$ on the last interval, $\frac{2n}{3(2\ell +1)}(1 + o_\ell(1))$ ways;
  \item $\pm1$ for all other replacements.
  \end{itemize}

  The residue set $\mathcal R:=\{R+1,R+3,\ldots,R+4\ell+1\}$ contains no two consecutive residues. Therefore, for each $t$, changes by $2t$ and $2t+1$ cannot both send a given residue into $\mathcal R$; the same holds for $-2t$ and $-(2t+1)$. Moreover, a change by $\pm1$ never sends a residue of $\mathcal R$ back into $\mathcal R$. Consequently,
  $$\Delta(K_3^{\Box n}[F_R]) \leq 2\ell \cdot \frac{2n}{3(2\ell+1)}(1 + o_\ell(1)) = \frac{2\ell }{3(2\ell + 1)}\cdot 2n(1 + o_\ell(1)).$$

  On the other hand, since $\frac{2\ell+1}{4\ell + 3} < \frac{1}{2} < \frac{2}{3}$,
  $$\bar a_\Box\bigg(\frac{2\ell+1}{4\ell + 3}\bigg) = 1 - \frac{4\ell +3}{3(2\ell + 1)} = \frac{2\ell}{3(2\ell+1)}.$$
\end{proof}

\subsection{The transformation $\alpha \to \frac{2}{3}\alpha$}\label{sec:transform23hamming}

\begin{Claim}\label{claim-mult-23}
  For each subset $F\subset \Z_3^{n}$ there exists a set $F' \subset \Z_3^{2n}$ such that $|F'| \geq \frac{2}{3}\cdot |F|\cdot 3^n$ and $\Delta(K_3^{\Box 2n}[F']) = \Delta(K_3^{\Box n}[F])$.
\end{Claim}

\begin{proof}
  Define $F''\subseteq\Z_3^{2n}$ by $(x,x')\in F''$ if and only if $x-x'\in F$, where $(x,x')$ denotes concatenation. For $z\in F''$, let $\deg_{+,F''}(z)$ and $\deg_{-,F''}(z)$ count, respectively, the neighbors of $z$ in $F''$ whose coordinate sum is one more or one less modulo~$3$. Then, for $(x,x')\in F''$,
  \begin{multline*}
  \deg_{+,F''}(x,x')
  =\sum_{i=1}^{n}\bigl(\mathbf{1}_{F''}(x+e_i,x')+\mathbf{1}_{F''}(x,x'+e_i)\bigr)\\
  =\sum_{i=1}^{n}\bigl(\mathbf{1}_{F}(x-x'+e_i)+\mathbf{1}_{F}(x-x'-e_i)\bigr)
  =\deg_F(x-x').
  \end{multline*}
  Similarly, $\deg_{-,F''}(x,x')=\deg_F(x-x')$. By averaging, for at least one $t\in\{0,1,2\}$,
  $$\bigl|\{z\in F'':\textstyle\sum_i z_i\not\equiv t\pmod3\}\bigr|
  \geq\frac23|F''|=\frac23|F|3^n.$$
  Let $F'$ be this set. Every edge of the Hamming graph changes the coordinate sum by $1$ or $-1$ modulo~$3$, so each vertex of $F'$ retains exactly one of the two directional degree classes above. Hence $\Delta(K_3^{\Box 2n}[F'])\leq\Delta(K_3^{\Box n}[F])$. Conversely, each $u\in F$ has lifts $(x,x')$ with $x-x'=u$ in every residue class of the total coordinate sum. A lift of a vertex attaining the maximum degree therefore belongs to $F'$, proving equality.
\end{proof}

For $\alpha\geq2/3$, we have $\bar a_{\Box}(2\alpha/3)=\bar a_{\Box}(\alpha)/2$. The claim therefore shows that an asymptotically optimal construction at density $\alpha$ yields one at density $2\alpha/3$.

In particular, from Statement~\ref{st-sequence1} we get 

\begin{Statement}\label{st-sequence1-23}
  For each integer $\ell \geq 1$ 

  $$\bar \Delta_\Box^*\bigg(\frac{2}{3}\bigg(1 - \frac{1}{2\ell+1}\bigg)\bigg) = \bar a_\Box\bigg(\frac{2}{3}\bigg(1 - \frac{1}{2\ell+1}\bigg)\bigg).$$
\end{Statement}

Recall that for the product case analogous transformation also exists and is much simpler, see Claim~\ref{cl:transform23tensor}.

\subsection{Examples stemming from ternary Hamming codes}\label{subsec-hamming-codes}

The following standard construction produces the well-known ternary Hamming codes and their dual simplex codes (see, e.g. \cite{HamCodes}). 

\begin{Statement}
  For every $\ell\geq1$ there exists a subspace $L_\ell$ of $\Z_3^{\frac{3^\ell - 1}{2}}$ of dimension $\ell$ such that for each $x\in L_\ell\backslash \{0\}$ it holds that $|x| = 3^{\ell -1}$. 
\end{Statement}

\begin{proof}
  We prove this by induction.

  Base: $\ell = 1$, $L_\ell = \Z_3$.

  Induction step: let $n=(3^{\ell-1}-1)/2$ and suppose that $L_{\ell-1}\subseteq\Z_3^n$ has the required properties. Define
  $$T:\Z_3^n\longrightarrow\Z_3^{3n+1},\qquad T(x)=(x,x,x,0),$$
  and let
  $$a_\ell=(0^n,1^n,(-1)^n,1).$$
  (Here and below, $c^r$ denotes the string in which the symbol $c$ is repeated $r$ times.) Set
  $$L_\ell:=T(L_{\ell-1})+\langle a_\ell\rangle.$$
  This subspace has dimension $\ell$. If $x\in L_{\ell-1}\setminus\{0\}$, then $T(x)$ has weight $3\cdot3^{\ell-2}=3^{\ell-1}$. For $c\in\{1,2\}$, every coordinate position among the first three blocks contributes exactly two nonzero symbols to $T(x)+ca_\ell$, and the last coordinate is nonzero; hence its weight is $2n+1=3^{\ell-1}$. Therefore every nonzero vector in $L_\ell$ has weight $3^{\ell-1}$.
\end{proof}

\begin{Statement}\label{st-perfect-ex}
  The set $L_\ell^{\perp}$ is a perfect dominating set in $K_3^{\Box N}$, where $N=(3^\ell-1)/2$. Consequently, for $m=3^N-3^{N-\ell}$ we have $\Delta_{K_3^{\Box N}}(m)=a_{K_3^{\Box N}}(m)$.
\end{Statement}

\begin{proof}
Set $N=(3^\ell-1)/2$, and let $A$ be the adjacency matrix of $K_3^{\Box N}$. The characteristic function $\chi$ of $L_\ell^{\perp}$ has Fourier decomposition $\chi = \frac{1}{3^{\ell}}\sum_{y\in L_\ell} u_y$. Consequently, the vector of the numbers of neighbours in $L_\ell^\perp$ is $$A\chi=\frac{1}{3^{\ell}}(2N\cdot u_0 + (2N-3\cdot3^{\ell-1})\sum_{y \in L_\ell\backslash \{0\}}u_y) = \frac{3^\ell - 1}{3^\ell}u_0 - \frac{1}{3^\ell}\sum_{y \in L_\ell\backslash \{0\}}u_y = u_0 - \chi,$$ which implies that $L_\ell^{\perp}$ is a perfect dominating set. For $F:=\Z_3^N\setminus L_\ell^{\perp}$, we therefore have $\Delta(K_3^{\Box N}[F])=2N-1=3^\ell-2$, while $a_{K_3^{\Box N}}(|F|)=(2-\frac{3^N}{3^N-3^{N-\ell}})(3^\ell-1) = (2 - \frac{1}{1-3^{-\ell}})\cdot(3^\ell - 1) = 3^\ell - 2$. 
\end{proof}

Claim~\ref{claim-mult-23} produces a set of at least
$$\frac{2}{3}\bigl(3^{3^\ell-1}-3^{3^\ell-1-\ell}\bigr)$$
vertices in dimension $3^\ell-1$ with the same maximum degree. Taking a subset of the indicated integer size, if necessary, and comparing with the average-degree lower bound yields equality at that size.

\subsection{Creating other (not necessarily optimal) asymptotic examples from existing ones}

\begin{Statement}
  Suppose that for $F \subset V(K_3^{\Box b})$ we have $|F| = \alpha \cdot 3^b$ and $\Delta(K_3^{\Box b}[F])=d$. Then for each natural number $p \geq 5$ coprime to $6$ 
  $$\bar\Delta_\Box^*\bigg(\frac{p-1}{p}\alpha\bigg) \leq \frac{p-2}{p-1}\frac{d}{2b}.$$
\end{Statement}

\begin{proof}
  Enumerate (arbitrarily) all $2(p-1)$ residues modulo $3p$ coprime to both $3$ and $p$ as $q_1,\ldots,q_{2(p-1)}$. Let $n$ be divisible by $2(p-1)b$. For $x\in \{0,1,2\}^n$ and $\ell \in [b]$ define $$s_\ell(x) = \sum_{j = 1}^{2(p-1)}q_j\cdot\sum_{i = \frac{n(\ell-1)}{b}+\frac{n(j-1)}{2(p-1)b}+1}^{\frac{n(\ell-1)}{b}+\frac{nj}{2(p-1)b}}x_i.$$ Let us call $[\frac{n(\ell-1)}{b}+1, \frac{n\ell}{b}]$ the $\ell$-th interval and $[\frac{n(\ell-1)}{b}+\frac{n(j-1)}{2(p-1)b}+1, \frac{n(\ell-1)}{b}+\frac{nj}{2(p-1)b}]$ the $j$-th subinterval of it. Let $F'_n$ consist of $x\in \{0, 1, 2\}^n$ such that $$s(x) := (s_1(x)\bmod 3,\ldots,s_b(x)\bmod 3) \in F,\text{ and } \sum_{\ell}s_\ell(x) \not \equiv R \mod p$$ for some residue $R$ and also there are $\frac{n}{6(p-1)b}(1+o(1))$ zeroes, ones and twos on each subinterval. By an averaging argument and the Chernoff-Hoeffding bound, $|F'_n| \geq \frac{p-1}{p}\cdot |F| \cdot 3^{n-b}\cdot(1-o(1)) = \frac{p-1}{p}\cdot \alpha \cdot 3^{n}\cdot(1-o(1))$ for at least one residue $R$.
  
  Consider a change of a coordinate on the $\ell$-th interval for some $x\in F'_n$. In order to give another sequence in $F'_n$ it should be consistent with some edge in $F$. Let us assume, without loss of generality, that this edge was determined by $s(x)\to s(x)+e_\ell$. Then the potential changes of the coordinate were $0 \to 1$, $1 \to 2$ or $2 \to 0$ on the $j$-th subinterval with $q_j \equiv 1\mod 3$ (there are $p-1$ such $j$) or $1\to 0$, $2 \to 1$ or $0 \to 2$ on the $j$-th subinterval with $q_j \equiv 2\mod 3$ (there are $p-1$ such $j$). For each non-zero residue $r$ modulo $p$ there is, by Chinese remainder theorem, exactly one $j$ for each of the following options:
  \begin{itemize}
  \item $q_j \equiv 1\mod 3$ and $q_j \equiv r \mod p$;
  \item $q_j \equiv 1\mod 3$ and $-2q_j \equiv r \mod p$;
  \item $q_j \equiv 2\mod 3$ and $-q_j \equiv r \mod p$;
  \item $q_j \equiv 2\mod 3$ and $2q_j \equiv r \mod p$. 
  \end{itemize}
  This shows that for each way to realize an edge in $F$, we cannot use one of the potential sub-intervals.
  Consequently, the typicality restriction can only decrease the number of available changes, and
  $$\deg_{F'_n}(x)\leq\deg_F(s(x))\cdot6(p-2)\frac{n}{6(p-1)b}(1+o(1)).$$
  Thus
  $$\Delta(K_3^{\Box n}[F_n']) \leq d\cdot (p-2)\frac{n}{(p-1)b}(1+o(1)) = 2n \cdot \frac{p-2}{p-1}\frac{d}{2b}(1+o(1)).$$
  This implies the statement. 
\end{proof}

Taking $b=1$ and $F = \Z_3$ gives Statement~\ref{st-sequence1} with $2\ell + 1$ coprime to 3. It also reproves Statement~\ref{st-sequence1-23} for $2\ell + 1$ coprime to 3 by taking $b = 1$, $F = \{0,1\}$.

\subsection{Examples stemming from the hypercube}

Let $n=kq$ and define $s_i(x):=\sum_{j=(i-1)q+1}^{iq}x_j\pmod 2$ for $i\in[k]$. Suppose that the $k$-dimensional cube has a vertex set $X$ of size $m_c$ whose induced subgraph has maximum degree $a_c$. Let $F\subseteq\{0,1,2\}^n$ consist of the strings $x$ such that $(s_1(x),\ldots,s_k(x))\in X$ and, in each block $x_{(i-1)q+1},\ldots,x_{iq}$, each of the symbols $0,1,2$ occurs at most
$$\frac{n}{3k}+\left(\frac{n}{k}\right)^{3/4}$$
times. By the preceding concentration estimate,
$$|F|\geq m_c\left(\left(\frac12-4e^{-2\sqrt{n/k}}\right)3^{n/k}\right)^k
\geq\frac{m_c}{2^k}3^n\left(1-8ke^{-2\sqrt{n/k}}\right)
=\frac{m_c}{2^k}3^n(1-o(1)).$$

On the other hand, $$\Delta(K_3^{\Box n}[F]) \leq \frac{2n}{3} + 2k\Big(\frac{n}{k}\Big)^\frac{3}{4} + a_c\Big(\frac{4n}{3k} + 4\Big(\frac{n}{k}\Big)^\frac{3}{4}\Big) = 2n\Big(\frac{1}{3}+ \frac{a_c}{k}\cdot\frac{2}{3}\Big)(1+o(1)).$$

Taking, for each $k$, a Huang-tight induced subgraph~\cite{CFGS} $X\subseteq Q_k$ with $|X|=2^{k-1}+1$ and maximum degree $a_c=O(\sqrt{k})$, and then letting $k\to\infty$, gives
$$\bar\Delta_\Box(1/2)=\bar\Delta_\Box^*(1/2)=\bar a_\Box(1/2).$$

\section{Open questions}\label{sec:open-q}

\subsection{The Hamming product}

Because the densities $\alpha$ for which we have proved $\bar\Delta^*_{\Box}(\alpha)=\bar a_{\Box}(\alpha)$ do not appear to share any exceptional feature, it is natural to conjecture that the same equality holds at every density.

\begin{Conjecture}
  $\bar \Delta_{\Box}(\alpha) = \bar a_{\Box}(\alpha)$ for all values $\alpha \in [0,1]$.
\end{Conjecture}

On the other hand, the range $\alpha \geq \frac{2}{3}$ seems similar to the interval $\alpha \geq \frac{1}{2}$ in the case of a hypercube. For instance, in both ranges the sets with minimum possible average degree of the induced subgraph are the complements of the independent sets (see the proof of Claim~\ref{cl1}).
Hence, it may be the case that the Hamming product has a (relatively small) jump at the point $\frac{2}{3}$. One small piece of supporting evidence is that, when $n=2$, every set of more than six vertices induces a subgraph of maximum degree~$4$.

\begin{Conjecture}
  There is a jump at $m=2\cdot3^{n-1}$ that is unbounded as a function of $n$, although it is $o(n)$ in the Hamming-product case.
\end{Conjecture}
Nevertheless, we cannot rule out the possibility that there are no jumps at all.

\subsection{The tensor product}

\begin{Question}
  What is the optimal constant in Theorem~\ref{thm-jump}? What is the optimal constant in Claim~\ref{cl:better-bound-bipart}? 
\end{Question}

All known low-degree examples for the tensor product with $1/3\leq\alpha\leq2/3$ are contained in a canonical bipartite subgraph. Moreover, the lower bounds for that bipartite subgraph are stronger than those for the whole graph. It is therefore natural to ask whether one may always restrict attention to such a subgraph.

\begin{Question}
  For each $3^{n-1}\leq m\leq2\cdot3^{n-1}$, is an optimal tensor-product example contained in a subgraph isomorphic to $\{0,1\}\times\Z_3^{n-1}$? 
\end{Question}

The statements proved in Subsections~\ref{ssec:no-big-perfect} and~\ref{ssec:no-small-perfect} suggest the following hypothesis.

\begin{Conjecture}
  There are only finitely many values $\alpha \in [\frac13, 1]$ such that there exists $n, m\in \Z_{>0}$ with $m = \alpha \cdot 3^{n}$ and $\Delta_{K_3^{\otimes n}}(m) = a_{K_3^{\otimes n}}(m)$.
\end{Conjecture}

\section*{Acknowledgements}
This paper was written as the core of my M.Sc. thesis.
I would like to thank my advisor, Ehud Friedgut, for his invaluable contributions to this thesis, and my parents and friends for their support.

\paragraph{AI usage declaration:} AI was used in this work for editing purposes only.

\section*{Appendices}

\appendix

\section{Analogues for $K_k^{\otimes n}$ and $K_k^{\Box n}$}\label{sec-gen-k}

In this appendix, we present analogues of several results above for an arbitrary integer $k \geq 3$ instead of $3$.

\subsection{Average degree}

\begin{Statement}
  At $\alpha=0$, both normalized quantities below are $0$. For $\alpha>0$, let $q\in\{0,1,\ldots,k-1\}$ be such that $q/k\leq\alpha\leq(q+1)/k$. Then
  $$\frac{a_{K_k^{\Box n}}(\alpha k^n)}{(k-1)n}
  =\frac{a_{K_k^{\otimes n}}(\alpha k^n)}{(k-1)^n}
  =\frac{q}{k-1}\left(2-\frac{q+1}{\alpha k}\right).$$
\end{Statement}

\begin{proof}
  For the lower bound, note that both the Hamming and tensor products contain a spanning subgraph $H$ isomorphic to the disjoint union of $k^{n-1}$ copies of $K_k$. In the Hamming product, these are the edges parallel to $e_1$; in the tensor product, they are the edges parallel to $(1,\ldots,1)^T$. Write $H=\bigsqcup_{j=1}^{k^{n-1}}H_j$, where each $H_j$ is a $k$-clique. 

Let $F\subseteq V(H)$. For $t=0,1,\ldots,k$, let
$$L_t:=\bigl|\{j\in[k^{n-1}]:|V(H_j)\cap F|=t\}\bigr|.$$ Then

$$k^{n-1} = \sum_{t = 0}^k L_t;$$
$$|F| = \sum_{t = 0}^k t \cdot L_t;$$
$$|E(H[F])| = \sum_{t = 0}^k \binom{t}{2}\cdot L_t.$$

For any $q=0,\ldots,k-1$,

$$|E(H[F])| - q|F| + \frac{q(q+1)}{2}\cdot k^{n-1} = \sum_{t = 0}^k L_t \cdot \frac{1}{2}(t^2 - (2q + 1)t + q(q+1)) = \sum_{t = 0}^k L_t \cdot \frac{1}{2}(t - q - 1)(t-q) \geq 0.$$

Consequently, $|E(H[F])| \geq q|F| - \frac{q(q+1)}{2}\cdot k^{n-1}$, and so $a_H(\alpha k^n) \geq q(2 - \frac{1}{\alpha}\frac{q+1}{k})$. Since $H$ is $(k-1)$-regular and both product graphs are edge-transitive, Claim~\ref{cl2} gives the lower bound for both graphs.

For the upper bound, write $\alpha k^n=qk^{n-1}+r$, where $0\leq r\leq k^{n-1}$. Take
  $$F=\bigl(\{0,\ldots,q-1\}\times\Z_k^{n-1}\bigr)\sqcup X\subseteq V(K_k^{\otimes n}),$$
  where the first term is interpreted as empty when $q=0$, and where $X\subseteq\{q\}\times\Z_k^{n-1}$ has size $r$. Then $$|E(K_k^{\otimes n}[F])| = k^{n-1}\cdot \frac{q(q-1)}{2}\cdot (k-1)^{n-1} + rq(k-1)^{n-1},$$
and so
$$\frac{a_{K_k^{\otimes n}}(\alpha\cdot k^n)}{(k-1)^n} \leq \frac{2q}{\alpha k(k-1)}\Big(\frac{q-1}{2} + \frac{r}{k^{n-1}}\Big) = \frac{q}{\alpha k(k-1)}(q-1 + 2(\alpha k - q)) = \frac{q}{k-1}\Big(2 - \frac{1}{\alpha}\frac{q+1}{k}\Big).$$

The analogous construction using the canonical $k$-partition works for the Hamming product.

\end{proof}

\begin{Notation}
  For real $\alpha\in[q/k,(q+1)/k]$ with $\alpha>0$, define $\bar a_k(\alpha):=\frac{q}{k-1}(2-\frac{q+1}{\alpha k})$, and set $\bar a_k(0):=0$. Note that $\bar a_k(\alpha) = \inf_{n, m\geq \alpha k^{n}} \frac{a_{K_k^{\otimes n}}(m)}{(k-1)^n} = \inf_{n, m\geq \alpha k^{n}} \frac{a_{K_k^{\Box n}}(m)}{(k-1) n}$.
\end{Notation}

\begin{Notation}
  Similarly, define $$\bar \Delta_{k,\Box}(\alpha) = \inf_{n, m\geq \alpha k^{n}} \frac{\Delta_{K_k^{\Box n}}(m)}{(k-1) n}, \,\, \bar \Delta_{k,\otimes}(\alpha) = \inf_{n, m\geq \alpha k^{n}} \frac{\Delta_{K_k^{\otimes n}}(m)}{(k-1)^n}.$$
\end{Notation}

\begin{Claim}
  In fact, $$\bar \Delta_{k,\Box}(\alpha) = \lim_{n \to \infty}\min_{m\geq \alpha k^{n}} \frac{\Delta_{K_k^{\Box n}}(m)}{(k-1) n}\text{ and } \bar \Delta_{k,\otimes}(\alpha) = \lim_{n\to\infty}\min_{ m\geq \alpha k^{n}} \frac{\Delta_{K_k^{\otimes n}}(m)}{(k-1)^n}.$$
\end{Claim}

\begin{proof}
  For the tensor product, if $F\subseteq V(K_k^{\otimes n})$, then
  $$\Delta\bigl(K_k^{\otimes(n+1)}[F\times\Z_k]\bigr)
  =(k-1)\Delta\bigl(K_k^{\otimes n}[F]\bigr).$$
  Hence the normalized minima form a nonincreasing sequence in $n$, and their limit equals their infimum.

  For the Hamming product, fix $n$ and $F\subseteq V(K_k^{\Box n})$. Given $n'$, write $n'=an+r$ with $0\leq r<n$, and partition $[n']$ into $n$ consecutive blocks $I_1,\ldots,I_n$, of which $r$ have size $a+1$ and the remaining $n-r$ have size $a$. Define
  $$\pi(x):=\left(\sum_{j\in I_1}x_j,\ldots,\sum_{j\in I_n}x_j\right)\in\Z_k^n,$$
  with all sums modulo~$k$, and set $F':=\pi^{-1}(F)$. Then $|F'|=k^{n'-n}|F|$. A change in one coordinate changes exactly one block sum, and each edge of $K_k^{\Box n}[F]$ has at most $a+1$ lifts incident with any fixed vertex. Therefore
  $$\Delta\bigl(K_k^{\Box n'}[F']\bigr)
  \leq(a+1)\Delta\bigl(K_k^{\Box n}[F]\bigr).$$
  After normalization, the right-hand side is $(1+o(1))\Delta(K_k^{\Box n}[F])/((k-1)n)$ as $n'\to\infty$. Thus the limsup is at most every fixed-dimensional normalized value, while the infimum is always at most the liminf. This proves the Hamming-product assertion.
\end{proof}

\subsection{Jumps in $K_k^{\otimes n}$ at points $\alpha = \frac{q}{k}$}

\begin{Claim}
  For each $q=1,2,\ldots,k$,
  $$\Delta_{K_k^{\otimes n}}(q\cdot k^{n-1}) = (q-1)\cdot(k-1)^{n-1},$$
  $$\Delta_{K_k^{\Box n}}(q\cdot k^{n-1}) = (q-1)n.$$
\end{Claim}

\begin{proof}
  The lower bounds follow from the average-degree formula. Matching examples are
  $$F_\Box:=\left\{x\in\Z_k^n:\sum_i x_i\pmod{k}\in\{0,\ldots,q-1\}\right\},
  \qquad
  F_\otimes:=\{0,\ldots,q-1\}\times\Z_k^{n-1}.$$
\end{proof}

\begin{Statement}
  For every $k\geq3$ there exists $c=c(k)>0$ such that, for each $q=1,\ldots,k-1$ and every $F\subseteq V(K_k^{\otimes n})$ with $|F|>qk^{n-1}$,
  $$\Delta(K_k^{\otimes n}[F])\geq(q-1)(k-1)^{n-1}+c(k-1)^n.$$
\end{Statement}

\begin{proof}

  Suppose that there is a family $F$ of size $q\cdot k^{n-1} + 1$ such that the maximum degree in the induced subgraph is at most $(q-1)\cdot (k-1)^{n-1} + d$. Denote $\alpha := \frac{|F|}{k^n}$, $\varepsilon := \frac{d}{(k-1)^n}$.
  
  The eigenvectors of the adjacency matrix $B_{k,n}$ of $K_{k}^{\otimes n}$ are the Fourier basis of $L_2(\Z_k^n)$, and the eigenvalue for $u_y$ ($u_y(x) = \omega_k^{x\cdot y}$, $\omega_k$ a primitive $k$th root of unity) is $(-1)^{|y|}(k-1)^{n-|y|}$. Let $\un_F = \sum_{y \in \Z_k^n}\hat{f}(y)u_y$ be the Fourier decomposition of the characteristic function of $F$. Define projections $P_\ell(\un_F) = \sum_{|y| = \ell}\hat{f}(y)u_y$ and $P_{\ge 2}(\un_F) = \sum_{|y| \ge 2}\hat{f}(y)u_y$.

  Then $$((q - 1)(k-1)^{n-1}+d)|F| \geq k^n|\hat f(0)|^2\cdot (k-1)^{n} - k^n||P_1(\un_F)||^2\cdot (k-1)^{n-1} - k^n||P_{\geq 2}(\un_F)||^2\cdot (k-1)^{n-3};$$

  $$\Big(\frac{q-1}{k-1} + \varepsilon\Big) \alpha \geq \alpha^2 - \frac{1}{k-1}(\alpha - \alpha^2 - ||P_{\ge 2}(\un_F)||^2) - \frac{1}{(k-1)^3}||P_{\ge 2}(\un_F)||^2;$$

  $$\varepsilon \alpha \geq \frac{\alpha}{k-1}(k\alpha - q) + \frac{k(k-2)}{(k-1)^3}||P_{\ge 2}(\un_F)||^2 \geq \frac{k(k-2)}{(k-1)^3}||P_{\ge 2}(\un_F)||^2,$$

  where the last inequality holds since $\alpha > \frac{q}{k}$.

  By Lemma~\ref{lemm}, there are a constant $K=K(k)$ and a Boolean function $g$ depending on at most one coordinate, which we may take to be the first coordinate, such that
$$
\|\un_F-g\|_2^2
<\frac{K\alpha\varepsilon (k-1)^3/[k(k-2)]}
{\alpha-\alpha^2-\alpha\varepsilon (k-1)^3/[k(k-2)]}
=\frac{K\varepsilon (k-1)^3/[k(k-2)]}
{1-\alpha-\varepsilon (k-1)^3/[k(k-2)]}
\leq C'(k)\varepsilon.
$$
Here $C'(k)>0$ is a constant, provided $n$ is sufficiently large (say, $1-\alpha>1/(2k)$) and $\varepsilon\leq c'(k)$ for a sufficiently small constant $c'(k)>0$.

  Assume without loss of generality that $|F\cap\{x_1=0\}|\geq|F\cap\{x_1=1\}|\geq\cdots\geq|F\cap\{x_1=k-1\}|$. For sufficiently large $n$ we have $\frac{q}{k} < \alpha < \frac{q}{k} + \frac{1}{2k}$, and so if $||g||^2_2 \neq \frac{q}{k}$, then $||\un_F-g||_2^2 \geq \frac{1}{2k}$, and so $\varepsilon \geq \frac{1}{2kC'(k)}$. Among $g$ with $||g||^2_2 = \frac{q}{k}$ the closest to $\un_F$ is clearly $g$ with $g(x_1=0)=\cdots=g(x_1=q-1)=1$ and $g(x_1=q)=\cdots=g(x_1=k-1)=0$, so we can now assume $||\un_F-g||_2^2 \leq C'(k)\varepsilon$ for this particular $g$.

  Let $F' := F\cap \{x_1\in\{0,\ldots,q-1\}\}$ and $F'' := F\cap \{x_1\in\{q,\ldots,k-1\}\}$; $\alpha' = \frac{|F'|}{q\cdot k^{n-1}}$ and $\alpha''=\frac{|F''|}{(k-q)k^{n-1}}$. Delete one element from $F$ so that still $\alpha'' > 0$ and $q\alpha' + (k-q)\alpha'' = q$. Consider the bipartite subgraph $G'$ of $K_k^{\otimes n}$ formed by all edges between $\{x_1\in\{0,\ldots,q-1\}\}$ and $\{x_1\in\{q,\ldots,k-1\}\}$. The matrix of this bipartite graph has singular values $\{\sqrt{q(k-q)}(k-1)^\ell\}_{\ell = 0}^{n-1}$ and the largest has multiplicity $1$. Then by the expander mixing lemma for the bipartite case:

  $$\Big\lvert \frac{e(F', F'')}{|E(G')|} - \alpha'\alpha''\Big\rvert \leq \frac{1}{k-1}\cdot \sqrt{\alpha'\alpha''(1-\alpha')(1-\alpha'')}.$$

  Consequently,

  $$\frac{((q-1)\cdot (k-1)^{n-1} + d)|F''|}{q(k-q)k^{n-1}(k-1)^{n-1}} \geq \alpha'\alpha'' - \frac{1}{k-1}\cdot \sqrt{\alpha'\alpha''(1-\alpha')(1-\alpha'')};$$

  $$\frac{q-1 + (k-1)\varepsilon}{q}\alpha'' \geq \alpha'\alpha'' - \frac{1}{k-1}\cdot \sqrt{\alpha'\alpha''\frac{\alpha''(k-q)}{q}(1-\alpha'')};$$

  $$\frac{q-1 + (k-1)\varepsilon}{q} \geq \alpha' - \frac{1}{k-1}\sqrt\frac{k-q}{q}.$$

  Since $||\un_F-g||_2^2 \leq C'(k)\varepsilon$, $\alpha' \geq 1 - C''(k)\varepsilon$, and so 

  $$\Big(\frac{k-1}{q} + C''(k)\Big)\varepsilon \geq \frac{1}{q} - \frac{1}{k-1}\sqrt\frac{k-q}{q}\geq \frac{1}{q}\Big(1 - \frac{k}{2(k-1)}\Big) > \frac{1}{4k},$$

  and so $\varepsilon \geq c(k)$.
  
\end{proof}

\paragraph{Additional examples for the tensor product}

\begin{Claim}
For each $t=2,3,\ldots,k-1$,
  $$\bar\Delta_{k,\otimes}\bigg(\frac{t\cdot(k-1)}{k^2}\bigg) = \bar a_k\bigg(\frac{t\cdot(k-1)}{k^2}\bigg) = \frac{(t-1)(k-2)}{(k-1)^2}.$$
\end{Claim}

\begin{proof}
  Let $F := \{0,1,\ldots,t-1\}\times(\Z_k\backslash \{k-1\})\times\Z_k^{n-2}$. Then $$\frac{\Delta(K_{k}^{\otimes n}[F])}{(k-1)^n} = \frac{(t-1)(k-2)(k-1)^{n-2}}{(k-1)^n} = \frac{(t-1)(k-2)}{(k-1)^2}.$$
  On the other hand, since $\frac{(k-1)t}{k^2} = \frac{t - \frac{t}{k}}{k} \in (\frac{t-1}{k}, \frac{t}{k}),$
  $$\bar a_k\bigg(\frac{(k-1)t}{k^2}\bigg) = \frac{t-1}{k-1}\Big(2 - \frac{k^2}{(k-1)t}\cdot\frac{t}{k}\Big) = \frac{(t-1)(k-2)}{(k-1)^2}.$$
\end{proof}

\subsection{Examples for the Hamming product}

Throughout this subsection, $o_k(1)$ denotes a quantity tending to zero as $n\to\infty$ with $k$ and any other displayed integer parameters fixed.

\begin{Statement}\label{st1}
  For each integer $2 \leq s < k$ there exists a sequence $m_n = \frac{1}{s}\cdot k^{n}(1-o_k(1))$ such that $$\frac{\Delta_{K_k^{\Box n}}(m_n)}{n(k-1)} \leq (1+o_k(1))\bar a_k\Big(\frac{1}{s}\Big),$$
  and consequently,
  $$\lim_{\alpha \to \frac{1}{s}-0}\bar\Delta_{k,\Box}(\alpha) = \bar a_k\Big(\frac{1}{s}\Big).$$
\end{Statement}

\begin{proof}

Write $k=as+r$, where $0\leq r<s$. By the Chernoff--Hoeffding bound, all but $o_k(k^n)$ strings $x\in\{0,\ldots,k-1\}^n$ satisfy
$$|\{i:x_i=t\}|\leq(1+o_k(1))\frac{n}{k}\qquad(t=0,1,\ldots,k-1).$$
By the pigeonhole principle, for some residue $R\pmod s$, at least $s^{-1}(1-o_k(1))k^n$ such strings have coordinate sum congruent to $R$. Denote this family by $F_{s,R}$. An edge inside $F_{s,R}$ replaces a symbol $t$ by $t+s\ell$ for some nonzero integer $\ell$. For each $\ell\in[-a,a]\setminus\{0\}$, there are $k-|\ell s|$ values of $t$ for which $t+s\ell\in[0,k-1]$. Thus every vertex of $F_{s,R}$ has degree at most 
$$2\sum_{\ell = 1}^a (k - \ell s)\frac{n}{k}(1+o_k(1)) = n\Big(2a - s\frac{a(a+1)}{k}\Big)(1+o_k(1)),$$
and thus
$$\frac{\Delta(K_k^{\Box n}[F_{s,R}])}{n(k-1)} \leq \frac{a}{k-1}\Big(2 - s\frac{a+1}{k}\Big)(1+o_k(1)).$$

On the other hand, $\frac{a}{k} \leq \frac{1}{s} < \frac{a+1}{k}$, (i.e., $q = a$) and thus $$\bar a_k\Big(\frac{1}{s}\Big) = \frac{a}{k-1}\Big(2 - s\frac{a+1}{k}\Big).$$

\end{proof}

\begin{Statement}
  For each integer $s$ such that $3 \leq 2s + 1 < 2k$ there exists a sequence $m_n = \frac{2}{2s+1}\cdot k^{n}(1-o_k(1))$ such that $$\frac{\Delta_{K_k^{\Box n}}(m_n)}{n(k-1)} \leq (1+o_k(1))\bar a_k\Big(\frac{2}{2s+1}\Big),$$
  and consequently,
  $$\lim_{\alpha \to \frac{2}{2s+1}-0}\bar\Delta_{k,\Box}(\alpha) = \bar a_k\Big(\frac{2}{2s+1}\Big).$$
\end{Statement}

\begin{proof}
  Let $a$ and $0\leq r<2s+1$ be such that $k=a(2s+1)+r$. Let $F_{s,R}$ be the family of all strings $x\in \{0,\ldots,k-1\}^n$ such that $\sum_{i}x_i \equiv R \mod 2s + 1$ or $\sum_{i}x_i \equiv R+s \mod 2s + 1$ and $|\{i: x_i = t\}| \leq (1+o_k(1))\frac{n}{k}$ for each $t = 0,1,\ldots,k-1$. By the Chernoff--Hoeffding bound and an averaging argument $|F_{s,R}| \geq \frac{2}{2s+1}k^{n}(1 - o_k(1))$ for at least one residue $R$. Then the degree of each element with residue of the sum equal to $R$ is at most
  $$\sum_{\ell : |(2s+1)\ell| < k,\, \ell \neq 0} (k - |(2s+1)\ell|)\frac{n}{k}(1+o_k(1)) + \sum_{\ell : |(2s+1)\ell+s| < k} (k - |(2s+1)\ell+s|)\frac{n}{k}(1+o_k(1)).$$
  This holds because to increase the sum by value $t$ ($|t| < k$, $t \neq 0$) going along an edge, we need to replace a value $x$ in some coordinate by value $x+t$, which can be done for exactly $k - |t|$ values $x$. Similarly, the degree of each element with residue of the sum equal to $R + s$ is at most
  $$\sum_{\ell : |(2s+1)\ell|<k,\;\ell\neq0} (k - |(2s+1)\ell|)\frac{n}{k}(1+o_k(1)) + \sum_{\ell : |(2s+1)\ell-s| < k} (k - |(2s+1)\ell-s|)\frac{n}{k}(1+o_k(1)).$$
  Note that by the change $\ell \to -\ell$ in the second summand the two estimates are equal. 
  
  Case 1: $r \leq s$. Then
  \begin{multline*}
  \Delta(K_k^{\Box n}[F_{s,R}]) \leq \Bigg(2\sum_{\ell = 1}^{a}(k - \ell (2s+1)) + \sum_{\ell = 0}^{a-1}(k - (\ell(2s+1) + s)) + \sum_{\ell = 1}^a(k - (\ell(2s+1) - s))\Bigg)\frac{n}{k}(1+o_k(1)) =\\= (4ka - (2s+1)a(2a+1))\frac{n}{k}(1+o_k(1)).
  \end{multline*}
  Thus, $$\frac{\Delta(K_k^{\Box n}[F_{s,R}])}{n(k-1)} \leq \frac{2a}{k-1}\bigg(2-\frac{2s+1}{2}\cdot\frac{2a+1}{k}\bigg)(1+o_k(1)).$$

  On the other hand, $\frac{2a}{k} \leq \frac{2}{2s+1} < \frac{2a + 1}{k}$ (i.e., $q = 2a$), and thus 
  $$\bar a_k\Big(\frac{2}{2s+1}\Big) = \frac{2a}{k-1}\Big(2 - \frac{2s+1}{2}\cdot\frac{2a+1}{k}\Big).$$

  Case 2: $r \geq s + 1$. Then 
  \begin{multline*}
  \Delta(K_k^{\Box n}[F_{s,R}]) \leq \Bigg(2\sum_{\ell = 1}^{a}(k - \ell (2s+1)) + \sum_{\ell = 0}^{a}(k - (\ell(2s+1) + s)) + \sum_{\ell = 1}^{a+1}(k - (\ell(2s+1) - s))\Bigg)\frac{n}{k}(1+o_k(1)) =\\= (2k(2a+1) - (2s+1)(a+1)(2a+1))\frac{n}{k}(1+o_k(1)).
  \end{multline*}
  Thus, $$\frac{\Delta(K_k^{\Box n}[F_{s,R}])}{n(k-1)} \leq \frac{2a+1}{k-1}\bigg(2-(2s+1)\frac{a+1}{k}\bigg)(1+o_k(1)).$$

  On the other hand, $\frac{2a+1}{k} \leq \frac{2}{2s+1} < \frac{2a + 2}{k}$ (i.e., $q = 2a+1$), and thus 
  $$\bar a_k\Big(\frac{2}{2s+1}\Big) = \frac{2a+1}{k-1}\Big(2 - \frac{2s+1}{2}\cdot\frac{2a+2}{k}\Big) = \frac{2a+1}{k-1}\Big(2 - (2s+1)\frac{a+1}{k}\Big).$$
\end{proof}

\begin{Statement}
  There exists a sequence $m_n = \frac{k-2}{k-1}\cdot k^{n}(1-o_k(1))$ such that $$\frac{\Delta_{K_k^{\Box n}}(m_n)}{n(k-1)} \leq (1+o_k(1))\bar a_k\Big(\frac{k-2}{k-1}\Big),$$
  and consequently,
  $$\lim_{\alpha \to \frac{k-2}{k-1}-0}\bar\Delta_{k,\Box}(\alpha) = \bar a_k\Big(\frac{k-2}{k-1}\Big).$$
\end{Statement}

\begin{proof}
  Let $F_R$ be the family of sequences $x\in \{0,\ldots,k-1\}^n$ such that $\sum_i x_i\not\equiv R\pmod{k-1}$ and $|\{i: x_i = t\}| \leq (1+o_k(1))\frac{n}{k}$ for each $t = 0,1,\ldots,k-1$. By the Chernoff--Hoeffding bound and an averaging argument $|F_{R}| \geq \frac{k-2}{k-1}k^{n}(1 - o_k(1))$ for at least one residue $R$. For each $x\in F_{R}$ with residue of the sum equal to $R'$ all edges lead to $F_R$ except those that increase the residue by $r := R - R' \not \equiv 0$ (we can assume $r \in [1,k-2]$). There are at least
  \begin{multline*}
  \sum_{\ell : |\ell (k-1) + r|<k} (k - |\ell (k-1) + r|)\frac{n}{k}(1 - o_k(1)) =\\= \sum_{\ell\in\{-1,0\}}(k - |\ell (k-1) + r|)\frac{n}{k}(1 - o_k(1)) = (k+1)\frac{n}{k}(1 - o_k(1))
  \end{multline*}
  such edges. Thus 

  $$\frac{\Delta(K_{k}^{\Box n}[F_R])}{(k-1)n} \leq \Big(1 - \frac{k+1}{k-1}\cdot \frac{1}{k}\Big)(1 + o_k(1)) = \frac{k^2 - 2k -1}{k(k-1)}(1 + o_k(1)).$$

  On the other hand, $\frac{k-2}{k} \leq \frac{k-2}{k-1} \leq \frac{k-1}{k}$ (i.e., $q = k-2$) and thus 
  $$\bar a_k\Big(\frac{k-2}{k-1}\Big) = \frac{k-2}{k-1}\Big(2 - \frac{k-1}{k-2}\cdot \frac{k-1}{k}\Big) = \frac{k^2 - 2k - 1}{k(k-1)}.$$
  
\end{proof}

\begin{Statement}
  There exists a sequence $m_n = \frac{k}{k+1}\cdot k^{n}(1-o_k(1))$ such that $$\frac{\Delta_{K_k^{\Box n}}(m_n)}{n(k-1)} \leq (1+o_k(1))\bar a_k\Big(\frac{k}{k+1}\Big),$$
  and consequently,
  $$\lim_{\alpha \to \frac{k}{k+1}-0}\bar\Delta_{k,\Box}(\alpha) = \bar a_k\Big(\frac{k}{k+1}\Big) = 1 - \frac{1}{k}.$$
\end{Statement}

\begin{proof}
  The proof is identical to that of the preceding statement.
\end{proof}

The first sequence of constructions satisfies the stronger conclusion below.

\begin{Statement}\label{st}
  For every integer $2\leq s<k$, there is no jump at $\alpha=1/s$; that is,
  
  $$\lim_{\alpha \to \frac{1}{s} + 0}\bar\Delta_{k,\Box}(\alpha) = \bar \Delta_{k,\Box}\Big(\frac{1}{s}\Big) = \bar a_k\Big(\frac{1}{s}\Big).$$
\end{Statement}

The proof uses the following two claims.

\begin{Claim}\label{cl11}
  For any $k \geq 3$ and $n \geq 1$ there exists a subset $D_{n,k} \subset V(K_k^{\Box n})$ such that $|D_{n,k}| = k^{n-1} + k - 2$ and $\Delta(K_k^{\Box n}[D_{n,k}])=k-2$.
\end{Claim}

\begin{proof}
  Denote $A_n^r := \{x \in \Z_{k}^{n} : \sum_{i} x_i \equiv r\pmod k\}$. Then each $A_n^r$ is an independent set in the Hamming product. Let us prove by induction on $n$ that there exists the desired vertex subset with the additional property that it is disjoint from $A_n^0$. 

  For the base case $n=1$, take $\{1,\ldots,k-1\}$. For the induction step, let $D_{n-1, k}$. Define $D_{n,k} = D_{n-1,k} \times \{0\} \cup A_{n-1}^0\times \{1,2,\ldots,k-1\}$. This clearly satisfies the maximum degree condition and its size is $|D_{n,k}| = |D_{n-1,k}| + (k-1)|A_{n-1}^0| = k^{n-2} + k - 2 + (k-1)k^{n-2} = k^{n-1} + k - 2$, as desired. 
\end{proof} 

\begin{Claim}\label{cl22}
  Fix integers $2\leq s<k$. For $r\in\Z_s$, let
  $$N_s^r:=\left|\left\{x\in\{0,1,\ldots,k-1\}^n:\sum_i x_i\equiv r\pmod{s}\right\}\right|.$$
  Then, uniformly in $r$ as $n\to\infty$,
  $$N_s^r=\frac{k^n}{s}(1+o(1)).$$
\end{Claim}

\begin{proof}

  Let $k = as + t$ with $0 \leq t < s$.

  For a nonnegative integer $S$, let $N(S):=|\{x\in\{0,1,\ldots,k-1\}^n:\sum_i x_i=S\}|$. Note that $(1+x+x^2+\cdots+x^{k-1})^n = \sum_{S\geq 0} N(S)x^S$. Let $\omega_s$ be a primitive $s$th root of unity. For each $\ell = 0,1,\ldots,s-1$ denote $W_\ell:=1+\omega_s^\ell+\omega_s^{2\ell}+\cdots+\omega_s^{(k-1)\ell}$. Then
  $$W_\ell^n = \sum_{r' = 0}^{s-1} N_s^{r'} \omega_s^{\ell r'}.$$

  Thus
  $$\sum_{\ell = 0}^{s-1}W_\ell^n \omega_s^{-\ell r} = sN_s^r +\sum_{\ell = 0}^{s-1}\sum_{r' \neq r}N_s^{r'} \omega_s^{\ell (r'-r)} = sN_s^r +\sum_{r' \neq r}N_s^{r'}\sum_{\ell = 0}^{s-1} \omega_s^{\ell (r'-r)} = sN_s^r.$$

  On the other hand, for any $\ell \not \equiv 0 \mod s$ (since $\sum_{i=0}^{s-1} \omega_s^{i\ell} = 0$)
  $$|W_\ell|=|\omega_s^{\ell as}+\omega_s^{\ell(as+1)}+\cdots+\omega_s^{\ell(k-1)}|\leq t\leq s-1$$
  and $W_0 = k$.

  Thus $$\frac{|N_s^r - \frac{1}{s}k^{n}|}{k^{n}} = \frac{|N_s^r - \frac{1}{s}W_0^n|}{k^{n}} \leq \frac{1}{s k^n}\sum_{\ell = 1}^{s-1}|W_\ell|^n \leq \frac{(s-1)^{n+1}}{s k^n} \to 0 \,\, (n \to \infty),$$
  as desired.
\end{proof}

\begin{proof}[Proof of Statement~\ref{st}]

First suppose that $s\geq3$. Fix some integer $b \geq 2$ and suppose $b | n$. Let $F$ be the set of $x \in \{0,1,\ldots,k-1\}^n$ such that
$$\left(\sum_{i=1}^{n/b}x_i,\sum_{i=n/b+1}^{2n/b}x_i,\ldots,\sum_{i=(b-1)n/b+1}^{n}x_i\right)\pmod s\in D_{b,s}$$ (where $D_{b, s}$ is as defined in Claim~\ref{cl11}) and in addition $|\{i : x_i = t\}| \leq (1+o_{k,b}(1))\frac{n}{k}$ for every $t=0,1,\ldots,k-1$. By Claim~\ref{cl22} and the Chernoff-Hoeffding bound $|F| = |D_{b,s}|s^{-b}k^n(1 - o_{k,b}(1))$. Each edge going from some vertex in $K_k^{\Box n}[F]$ corresponds to changing the value in some coordinate, and this change should either conserve the residue modulo $s$ or change it in accordance to the edges of $D_{b,s}$. It is shown in the proof of Statement~\ref{st1} that there are $\bar a_k(\frac{1}{s})(k-1)n(1 + o_{k,b}(1))$ edges of the first type. Since at most $(\frac{k}{s} + 1) \leq \frac{2k}{s}$ changes of value in range $[0,k-1]$ change the residue in the same way, there are at most $\Delta(D_{b,s}) \cdot \frac{2k}{s} \cdot\frac{n}{b}$. The fraction of vertices inside $F$ is then $$\alpha = |D_{b,s}|s^{-b}(1 - o_{k,b}(1)) = \frac{1}{s}\Big(1 + \frac{s-2}{s^{b-1}}\Big)(1 - o_{k,b}(1)),$$

and the maximum degree normalized by $n(k-1)$ is at most
$$\bar a_k\Big(\frac{1}{s}\Big)(1 + o_{k,b}(1)) + \frac{(s-2)\cdot 2k}{sb(k-1)}.$$

Then for any $\varepsilon > 0$ and any $b \geq 3$

$$\bar \Delta_{k,\Box}\Big(\frac{1}{s}(1 + \frac{s-2}{s^{b-1}})(1-\varepsilon)\Big) \leq \bar a_k\Big(\frac{1}{s}\Big) + \frac{(s-2)\cdot 2k}{sb(k-1)}.$$

Substituting an appropriate $\varepsilon$ we get

$$\bar \Delta_{k,\Box}\Big(\frac{1}{s}(1 + \frac{s-2}{2s^{b-1}})\Big) \leq \bar a_k\Big(\frac{1}{s}\Big) + \frac{(s-2)\cdot 2k}{sb(k-1)},$$

which proves the statement (considering $b \to \infty$).

For $s=2$, replace $D_{b,s}$ by a Huang-tight induced subgraph~\cite{CFGS} of the hypercube $Q_b$ on $2^{b-1}+1$ vertices and with maximum degree $O(\sqrt b)$. The same argument then applies, since the additional normalized degree tends to zero as $b\to\infty$.
  
\end{proof}

\begin{Statement}
  For any $s > k$ coprime to $k!$ we have 

  $$\lim_{\alpha \to 1-\frac{1}{s} - 0}\bar\Delta_{k,\Box}(\alpha) = \bar a_k\Big(1-\frac{1}{s}\Big).$$

  If $k$ is a prime, then also
  for every $\ell\in[k-1]$,
  $$\lim_{\alpha \to \frac{\ell}{k}(1-\frac{1}{s}) - 0}\bar\Delta_{k,\Box}(\alpha) = \bar a_k\Big(\frac{\ell}{k}\Big(1-\frac{1}{s}\Big)\Big).$$
\end{Statement}

\begin{proof}
  Fix $s>k$ with $\gcd(s,k!)=1$, and let $n$ be a multiple of $s-1$. Partition $[n]$ into consecutive blocks $I_1,\ldots,I_{s-1}$ of equal size, and define
  $$W(x):=\sum_{j=1}^{s-1}j\sum_{i\in I_j}x_i\pmod{s}.$$
  Choose a sequence $\eta_n\to0$ slowly enough that all but $o(k^n)$ strings satisfy
  $$\left|\,|\{i\in I_j:x_i=t\}|-\frac{n}{(s-1)k}\right|\leq\eta_n n
  \qquad(j\in[s-1],\ t\in\Z_k).$$
  Among these typical strings, choose a residue $R\in\Z_s$ for which the class $W(x)=R$ has size at most one $s$th of the total, and let $F$ be the typical strings with $W(x)\ne R$. Then
  $$|F|\geq\left(1-\frac1s\right)k^n(1-o(1)).$$

  Fix $x\in F$, write $W(x)=R'\ne R$, and consider replacing a symbol $a$ by $b\ne a$ in block $I_j$. The new weighted residue is $R'+(b-a)j$. Since $1\leq|b-a|\leq k-1$ and $\gcd(s,k!)=1$, the element $b-a$ is invertible modulo~$s$. Thus, for each ordered pair $(a,b)$, exactly one block $I_j$ would send the weighted residue to the excluded value $R$. Ignoring the typicality condition can only increase the degree, so the block-balance estimates give
  $$\deg_F(x)\leq\left(1-\frac{1}{s-1}\right)n(k-1)(1+o(1)).$$
  Hence
  $$\frac{\Delta(K_k^{\Box n}[F])}{n(k-1)}\leq\left(1-\frac{1}{s-1}\right)(1+o(1)).$$
  Finally, $1-1/s>(k-1)/k$, so the relevant interval has $q=k-1$, and
  $$\bar a_k\left(1-\frac1s\right)=2-\frac{1}{1-1/s}=1-\frac{1}{s-1}.$$

  Now assume that $k$ is prime. Let $F'\subseteq\Z_k^{n'}$ induce maximum degree $d$, and put $n=(k-1)n'$. For $x=(x_1,\ldots,x_{k-1})\in(\Z_k^{n'})^{k-1}$, define
  $$\psi(x):=x_1+2x_2+\cdots+(k-1)x_{k-1},
  \qquad F:=\psi^{-1}(F').$$
  The map $\psi$ is surjective and has kernel of size $k^{(k-2)n'}$, so $|F|=k^{(k-2)n'}|F'|$. For each nonzero $t\in\Z_k$, the number of neighbors of $x$ in $F$ whose total coordinate sum is $t$ more than that of $x$ is
  $$\sum_{b=1}^{k-1}\sum_{i=1}^{n'}\mathbf{1}_{F'}(\psi(x)+bt e_i)=\deg_{F'}(\psi(x)),$$
  because multiplication by $t$ permutes the nonzero elements of $\Z_k$.

  Order the residue classes of the total coordinate sum so that
  $$|F\cap\{\textstyle\sum_i x_i=r_1\}|\geq\cdots\geq|F\cap\{\textstyle\sum_i x_i=r_k\}|,$$
  and let $F_\ell$ be the union of the first $\ell$ classes. Then $|F_\ell|\geq(\ell/k)|F|$ and $\Delta(K_k^{\Box n}[F_\ell])\leq(\ell-1)d$. Consequently,

  $$\bar\Delta_{k,\Box}\Big(\frac{\ell}{k}\alpha\Big) \leq \frac{\ell - 1}{k-1}\bar\Delta_{k,\Box}(\alpha).$$

  If $(k-1)/k<\alpha<1$, then $(\ell-1)/k<(\ell/k)\alpha<\ell/k$, and therefore $\bar a_k(\frac{\ell}{k}\alpha) = \frac{\ell-1}{k-1}(2 - \frac{1}{\frac{\ell}{k}\alpha}\frac{\ell}{k}) = \frac{\ell-1}{k-1}\bar a_k(\alpha)$. For $\alpha = 1 - \frac{1}{s}$ with $s > k$ the above holds and so we get 

  $$\lim_{\alpha \to \frac{\ell}{k}(1-\frac{1}{s}) - 0}\bar\Delta_{k,\Box}(\alpha) = \bar a_k\Big(\frac{\ell}{k}\Big(1-\frac{1}{s}\Big)\Big).$$
  
\end{proof}

\section{Asymptotics as $k\to \infty$}

\subsection{The tensor product}

The proof closely follows the argument of~\cite{KneserCEFL}.

\begin{Theorem}
  Let $s\geq1$ be an integer, let $k\geq1000s^2$, and let $s<\lambda\leq s+1$. If $n\geq3$, then
  $$\Delta_{K_k^{\otimes n}}(\lambda k^{n-1})
  =\frac{\lambda s}{s+1}(k-1)^{n-1}\bigl(1\pm O(\sqrt{s/k})\bigr).$$
  If $n=2$, then
  $$\Delta_{K_k^{\otimes 2}}(\lambda k)
  =\frac{\lambda s}{s+1}(k-1)\bigl(1\pm O(\sqrt{s/k}+\sqrt{\ln k/k})\bigr).$$
\end{Theorem}

The proof is given in the following lemmas. Throughout this subsection, we assume $k\geq1000s^2$.

\begin{Lemma}\label{lem-big-comp}

Let $\Omega:=\Z_{k_1}\times\cdots\times\Z_{k_n}$, equipped with the uniform probability measure, and let $F\subseteq\Omega$. Write the Fourier decomposition of its characteristic function as
  $$f(x)=\sum_{y\in\Omega}\hat f(y)\prod_{i=1}^n\omega_{k_i}^{x_i y_i},$$
  where each $\omega_{k_i}$ is a primitive $k_i$th root of unity. Set $\alpha:=|F|/|\Omega|=\hat f(0)$, and let $e_i$ denote the $i$th standard basis vector. Let $f_1$ be the projection of $f$ onto the Fourier characters of Hamming weight one; equivalently,
  $$f_1=\sum_{i=1}^n\sum_{\ell\in\Z_{k_i}\setminus\{0\}} \hat f(\ell e_i)\omega_{k_i}^{\ell x_i}.$$
  If $\alpha=0$, set $\eta:=0$; otherwise, define $\eta$ by $\alpha\eta:=\|f_1\|_2^2$. Define $\gamma_{i,a} = \frac{|F\cap \{x_i = a\}|}{|\{x_i = a\}|}$ for $a\in \Z_{k_i}$ and $\gamma'_{i,a} = \gamma_{i,a}-\alpha$. Then $$\max_{i,a}(\gamma_{i,a}')^2 \geq \eta^3 - 3\alpha\eta.$$

\end{Lemma}

\begin{proof}

Let $\hat{f}_{i,\ell} := \hat{f}(\ell e_i)$ for $\ell \neq 0$ and define $\hat f_{i,0} = 0$ for all $i\in [n]$.

Then for $\ell \neq 0$ 

\begin{equation*}
  \hat f_{i,\ell} = \hat f(\ell e_{i}) = \mathbb{E}(f\omega_{k_i}^{-\ell x_i}) = \sum_{a = 0}^{k_i-1} \frac{\gamma_{i,a}}{k_i}\omega_{k_i}^{-a\ell} = \sum_{a = 0}^{k_i-1} \frac{\gamma'_{i,a}}{k_i}\omega_{k_i}^{-a\ell}.
\end{equation*}

Also, $\hat f_{i,0}=0$ and $\sum_{a=0}^{k_i-1}\gamma'_{i,a}=0$.
Thus for all $\ell \in \Z_{k_i}$ we have $\hat f_{i,\ell} = \sum_{a = 0}^{k_i-1} \frac{\gamma'_{i,a}}{k_i}\omega_{k_i}^{-a\ell}$.

\begin{multline*}
  \alpha\eta=\|f_1\|_2^2 = \sum_{i,\ell\neq 0}|\hat f(\ell e_i)|^2 = \sum_{i,\ell} |\hat f_{i,\ell}|^2 = \sum_{i,\ell, a_1,a_2}\frac{1}{k_i^2}\gamma_{i,a_1}'\gamma_{i,a_2}'\omega_{k_i}^{-\ell(a_1-a_2)} = \sum_{i,a}\frac{1}{k_i}(\gamma_{i,a}')^2.
\end{multline*}

\begin{multline*}
  \mathbb{E}f_1^4 = \mathbb{E}\Big(\sum_{i=1}^n\sum_{\ell\in\Z_{k_i}\setminus\{0\}}\hat f(\ell e_i)\omega_{k_i}^{\ell x_i}\Big)^4 =\\= 3\sum_{i\neq j}\sum_{\ell_1, \ell_2\neq 0}\hat f(\ell_1 e_{i}) \hat f(-\ell_1 e_{i}) \hat f(\ell_2 e_{j}) \hat f(-\ell_2 e_{j}) + \\+\sum_{i}\sum_{\ell_1+\ell_2+\ell_3+\ell_4=0,\;\ell_1,\ldots,\ell_4\neq0} \hat f(\ell_1 e_{i}) \hat f(\ell_2 e_{i}) \hat f(\ell_3 e_{i}) \hat f(\ell_4 e_{i})
\end{multline*}

\begin{equation*}
  \sum_{i\neq j}\sum_{\ell_1, \ell_2\neq 0}\hat f(\ell_1 e_{i}) \hat f(-\ell_1 e_{i}) \hat f(\ell_2 e_{j}) \hat f(-\ell_2 e_{j}) \leq \Big(\sum_{i,\ell}|\hat f(\ell e_i)|^2\Big)^2 = \alpha^2\eta^2;
\end{equation*}

\begin{multline*}
  \sum_{i}\sum_{\ell_1+\ell_2+\ell_3+\ell_4=0,\;\ell_1,\ldots,\ell_4\neq0} \hat f(\ell_1 e_{i}) \hat f(\ell_2 e_{i}) \hat f(\ell_3 e_{i}) \hat f(\ell_4 e_{i}) = \sum_i \sum_{\ell_1, \ell_2, \ell_3}\hat f_{i, \ell_1}\hat{f}_{i,\ell_2}\hat{f}_{i,\ell_3}\hat{f}_{i, -\ell_1 - \ell_2 - \ell_3} =\\ = \sum_{i}\frac{1}{k_i^4}\sum_{\ell_1, \ell_2, \ell_3, a_1,a_2,a_3,a_4} \gamma'_{i,a_1}\gamma'_{i,a_2}\gamma'_{i,a_3}\gamma'_{i,a_4}\omega_{k_i}^{-\ell_1(a_1 - a_4)}\omega_{k_i}^{-\ell_2(a_2 - a_4)}\omega_{k_i}^{-\ell_3(a_3 - a_4)} =\\= \sum_{i} \frac{1}{k_i}\sum_{a}(\gamma'_{i,a})^4 \leq \max_{i,a}(\gamma_{i,a}')^2\sum_{i,a}\frac{1}{k_i}(\gamma_{i,a}')^2 = \alpha\eta\cdot \max_{i,a}(\gamma_{i,a}')^2
\end{multline*}

By H\"older's inequality, $\mathbb{E}(ff_1)\leq\|f\|_{4/3}\|f_1\|_4$. Since $f_1$ is the orthogonal projection of $f$ onto the first Fourier level, $\mathbb{E}(ff_1)=\|f_1\|_2^2=\alpha\eta$. Consequently,
$$\eta^4\alpha^4\leq\alpha^3\mathbb{E}f_1^4.$$

Therefore,

$$\eta^4 \alpha \leq 3\alpha^2\eta^2 + \alpha\eta\cdot \max_{i,a}(\gamma_{i,a}')^2.$$

If $\alpha\eta=0$, the conclusion is immediate. Otherwise, division by $\alpha\eta$ gives
$$\max_{i,a}(\gamma_{i,a}')^2\geq\eta^3-3\alpha\eta.$$
\end{proof}

\begin{Lemma}\label{lem:outdeg}
  Suppose that $|F|=s k^{n-1}$. Let $\mathcal B := [0,k-a_1-1]\times\cdots\times[0,k-a_{\ell-1}-1] \times [0, k-a_\ell-1]\times [0,k-1]^{n-\ell}$ where $a := \sum_{p} a_p \leq s$ and $a_p > 0$; let $F' = F\cap \mathcal B$ and $F'' = F\setminus F'$. Let $\beta := \frac{|F'|}{|\mathcal B|} \geq \frac{|F'|}{k^n} =: \beta'$ and $\alpha'' = \frac{|F''|}{k^n}$. Then 
  $$e(F'', F') \geq |\mathcal{B}|(k-1)^{n-1}(1 - \frac{s}{k-1})\Big(k\alpha'' - \frac{\sqrt{sk}}{k-s-1}\Big)\beta.$$
\end{Lemma}

\begin{proof}
  
Let $\mathcal{A}_p = [0,k-a_1-1]\times\cdots\times[0,k-a_{p-1}-1] \times [k-a_p, k-1]\times [0,k-1]^{n-p}$ for each $p \in [\ell]$. Note that $\mathcal{A}_p$ are pairwise disjoint and $[0,k-1]^n\setminus \mathcal{B} = \bigsqcup_{p = 1}^{\ell}\mathcal{A}_p$. The maximal bipartite subgraph of $K_k^{\otimes n}$ with parts $\mathcal{A}_p$ and $\mathcal{B}$ has matrix 
\begin{multline*}
  A_p = (J_{k-a_1}-E_{k-a_1})\otimes\cdots\otimes (J_{k-a_{p-1}}-E_{k-a_{p-1}})\otimes J_{a_p,k-a_p}\otimes \\\otimes \begin{pmatrix}
  J_{k-a_{p+1}} - E_{k-a_{p+1}}\\
  J_{a_{p+1}, k-a_{p+1}}
\end{pmatrix} \otimes\cdots\otimes \begin{pmatrix}
  J_{k-a_{\ell}} - E_{k-a_{\ell}}\\
  J_{a_{\ell}, k-a_{\ell}}
\end{pmatrix}\otimes (J_k-E_k)^{\otimes (n-\ell)}
\end{multline*}
where $E_a$ is the $a\times a$ identity matrix, $J_a$ is the $a\times a$ all-ones matrix, and $J_{a,b}$ is the $a\times b$ all-ones matrix.
Then 
\begin{multline*}
  A_p^T A_p = (E_{k-a_1} + (k-a_1-2)J_{k-a_1})\otimes\cdots\otimes (E_{k-a_{p-1}} + (k-a_{p-1}-2)J_{k-a_{p-1}})\otimes a_p J_{k-a_p}\otimes\\ \otimes (E_{k-a_{p+1}}+(k-2)J_{k-a_{p+1}})\otimes\cdots\otimes(E_{k-a_{\ell}}+(k-2)J_{k-a_{\ell}})\otimes\bigl(E_k+(k-2)J_k\bigr)^{\otimes(n-\ell)}.
\end{multline*} Consequently, the nonzero singular values of $A_p$ are $$\prod_{i = 1}^{p-1}(k-a_i-1)^{\epsilon_i}\cdot \sqrt{a_p(k-a_p)}\cdot \prod_{i=p+1}^\ell \sqrt{1 + (k-2)(k-a_i)}^{\epsilon_i}\cdot (k-1)^d, \,\epsilon_i \in \{0,1\},\, d\in\{0,\ldots,n-\ell\}.$$ The largest singular value $\sigma_1$ has multiplicity $1$, and so the second largest singular value is $\sigma_2 = \frac{\sigma_1}{\min(\min_{i=1}^{p-1}(k-a_i-1), \min_{i=p+1}^{\ell}\sqrt{1 + (k-2)(k-a_i)}, k-1)}\leq \frac{\sigma_1}{k-s-1}$. Also 

\begin{multline*}
  e(\mathcal{A}_p, \mathcal{B}) = \prod_{i=1}^{p-1}(k-a_i)(k-a_i-1)\cdot a_p(k-a_p)\cdot\prod_{i=p+1}^{\ell}(k-a_i)(k-1)\cdot(k(k-1))^{n-\ell} \geq \\ \ge a_p|\mathcal B|(k-1)
^{n-1}\prod_{i=1}^{\ell}(1-\frac{a_i}{k-1}) \geq a_p|\mathcal{B}|(k-1)
^{n-1}(1 - \frac{\sum_{i}a_i}{k-1}) \geq\\\geq a_p|\mathcal{B}|(k-1)
^{n-1}(1 - \frac{s}{k-1})
\end{multline*}

Let $F_p:=F''\cap\mathcal A_p$ and $\alpha_p := \frac{|F_p|}{|\mathcal A_p|}$. Then by the expander mixing lemma for bipartite graphs, 

$$\frac{e(F_p, F')}{e(\mathcal{A}_p, \mathcal{B})} \geq \alpha_p\beta - \frac{\sigma_2}{\sigma_1}\sqrt{\alpha_p \beta (1-\alpha_p)(1-\beta)} \ge \alpha_p\beta - \frac{1}{k-s-1}\sqrt{ \beta (1-\alpha_p)}.$$

Consequently, $$e(F'', F') \geq |\mathcal{B}|(k-1)^{n-1}(1 - \frac{s}{k-1})\Big(\beta\sum_{p = 1}^{\ell}a_p \alpha_p - \frac{\sqrt{\beta}}{k-s-1}\sum_{p = 1}^{\ell}a_p\sqrt{(1-\alpha_p)}\Big).$$

Note that $\sum_{p=1}^\ell a_p \alpha_p \geq \frac{\sum_{p}|F_p|}{k^{n-1}} = \frac{|F| - |F'|}{k^{n-1}}= k\alpha'' = s - k\beta' \geq s - k\beta \geq a - k\beta$. 

Since $x\mapsto\sqrt{x}$ is concave,
$$
\sum_p a_p\sqrt{1-\alpha_p}
\leq a\sqrt{\sum_p\frac{a_p}{a}(1-\alpha_p)}
=a\sqrt{1-\frac{\sum_p a_p\alpha_p}{a}}
\leq\sqrt{ak\beta}
\leq\sqrt{sk\beta}.
$$

Consequently, $$e(F'', F') \geq |\mathcal{B}|(k-1)^{n-1}(1 - \frac{s}{k-1})\Big(k\alpha'' - \frac{\sqrt{sk}}{k-s-1}\Big)\beta,$$
as desired. 

\end{proof}

\begin{Lemma}\label{lem-frac-get-bigger}
  Suppose that $|F| = s \cdot k^{n-1}$, $\Delta(K_k^{\otimes n}[F]) \leq \frac{s^2}{s+1}(k-1)^{n-1}$ and for some $i \in [n]$ and $b \in \Z_k$ we have $\gamma = \frac{|F\cap \{x_i = b\}|}{k^{n-1}} \geq c > 0$. Then $\gamma \geq \frac{s}{s+1} - \sqrt{\frac{s}{c(k-1)}}$.
\end{Lemma}

\begin{proof}
  Let $F':=F\cap\{x_i=b\}$ and $F'':=F\setminus F'$, and set $\gamma'':=|F''|/((k-1)k^{n-1})$. Then $\gamma+(k-1)\gamma''=|F|/k^{n-1}=s$. Apply the bipartite expander mixing lemma to the subgraph with parts $\{x_i=b\}$ and $\{x_i\neq b\}$. Its largest and second-largest singular values have ratio $1/(k-1)$, and it has $(k-1)^n k^{n-1}$ edges. We obtain 

  $$e(F', F'') \geq (k-1)^{n}k^{n-1}(\gamma\gamma'' - \frac{1}{k-1}\sqrt{\gamma\gamma''(1-\gamma)(1-\gamma'')}) \geq |F'|(k-1)^{n-1}(\gamma''(k-1) - \sqrt{\gamma''/\gamma}), $$
  and so $$\frac{s^2}{s+1} \geq \frac{\Delta(K_k^{\otimes n}[F])}{(k-1)^{n-1}} \geq \gamma''(k-1) - \sqrt{\gamma''/\gamma} \geq s - \gamma - \sqrt{\frac{s}{c(k-1)}}.$$
  This proves the desired bound.
\end{proof}

\begin{Lemma}\label{lem-many-dense-stars}
  Suppose that $|F|=s k^{n-1}$ and $\Delta(K_k^{\otimes n}[F])\leq\frac{s^2}{s+1}(k-1)^{n-1}$. Then there are at least $s$ pairs $(i_j,q_j)\in[n]\times\Z_k$ such that
  $$\frac{|F\cap\{x_{i_j}=q_j\}|}{k^{n-1}}\geq\frac{s}{s+1}-4\sqrt{\frac{s}{k-1}}.$$
\end{Lemma}

\begin{proof}

Fix any $s-1$ pairs $\{(i_j,q_j)\}$. We will find a different pair $(i,q)$ such that
$$\frac{|F\cap\{x_i=q\}|}{k^{n-1}}\geq\frac{s}{s+1}-4\sqrt{\frac{s}{k-1}}.$$

Let $\mathcal B=\{x:x_{i_j}\neq q_j\ \text{for every }j\}$. We can without loss of generality assume that $\mathcal B$ is as in Lemma~\ref{lem:outdeg} with $a = s - 1$. From now on we use the notation from this lemma. Let $f'$ be the characteristic function of $F'$ in $\mathcal B$. Set $\|f'_1\|_2^2=\eta\beta$. Then 
\begin{multline*}
  e(F',F') \geq \prod_{p = 1}^\ell (k-a_p)\cdot k^{n - \ell} \prod_{p = 1}^\ell (k-1-a_p)\cdot (k-1)^{n - \ell}(\beta^2 - \frac{1}{k-1-\max a_p}\eta \beta - \frac{1}{(k-1-\max a_p)^2}\beta) \geq \\ \geq |\mathcal B|(k-1)^{n}(1 - \frac{s}{k-1})(\beta - \frac{1}{k-1-s}\eta - \frac{1}{(k-1-s)^2})\beta \geq\\\geq |\mathcal B|(k-1)^{n}(1 - \frac{s}{k-1})(\beta' - \frac{1}{k-1-s}\eta - \frac{1}{(k-1-s)^2})\beta.
\end{multline*}

Thus 
\begin{multline*}
  e(F',F) \geq |\mathcal B|(k-1)^{n-1}(1 - \frac{s}{k-1})(\beta'(k-1) - \eta\frac{k-1}{k-s-1} - \frac{k-1}{(k-s-1)^2} + k\alpha'' - \frac{\sqrt{sk}}{k-s-1})\beta \geq\\\geq |F'|(k-1)^{n-1}(k\alpha - \frac{4s^2}{k} - 3\sqrt{\frac{s}{k}} - \eta) = |F'|(k-1)^{n-1}(s - \frac{4s^2}{k} - 3\sqrt{\frac{s}{k}}- \eta).
\end{multline*}
Consequently, 
$$\frac{s^2}{s+1}(k-1)^{n-1} \geq \Delta(K_{k}^{\otimes n}[F]) \geq (k-1)^{n-1}(s - \frac{4s^2}{k} - 3\sqrt{\frac{s}{k}}- \eta),$$
and so 
$$\eta \geq \frac{s}{s+1} - \frac{4s^2}{k} - 3\sqrt{\frac{s}{k}} \geq \frac{1}{4}.$$

Since $|\mathcal B|\geq(1-(s-1)/k)k^n$, we have
$$\beta=\frac{|F'|}{|\mathcal B|}\leq\frac{s}{k-s+1}\leq\frac{1}{999}.$$
Lemma~\ref{lem-big-comp} therefore gives $i\in[n]$ and $b\in\Z_{k_i}$, where $k_i=k-a_i$ for $i\leq\ell$ and $k_i=k$ otherwise, such that
$$|\gamma'_{i,b}|\geq\sqrt{\frac{1}{64}-\frac{3\beta}{4}}\geq\frac19.$$
Because $\gamma'_{i,b}=\gamma_{i,b}-\beta\geq-\beta>-1/9$, this deviation must be positive. Hence $\gamma_{i,b}\geq\gamma'_{i,b}\geq1/9$, and
$$\frac{|F\cap\{x_i=b\}|}{k^{n-1}}
\geq\frac{\gamma_{i,b}|\mathcal B|}{k_i k^{n-1}}
\geq\frac19\left(1-\frac{s}{k}\right)\geq\frac1{16}.$$
Lemma~\ref{lem-frac-get-bigger} now yields the desired bound.

\end{proof}

\begin{Lemma}\label{lem-jumps}
$$\Delta_{K_k^{\otimes n}}(s\cdot k^{n-1}+1) \geq (k-1)^{n-1}\Bigg(\frac{s^2}{s+1} - 7\sqrt{\frac{s^3}{k-1}}\Bigg)$$
\end{Lemma}

\begin{proof}
  Let $F_+$ have size $sk^{n-1}+1$ and suppose that $\Delta(K_k^{\otimes n}[F_+])\leq\frac{s^2}{s+1}(k-1)^{n-1}$. Delete one vertex and denote the resulting set by $F$. If $F$ were isomorphic to $[0,s-1]\times\Z_k^{n-1}$, then the deleted vertex would have $s(k-1)^{n-1}$ neighbors in $F$, contradicting the assumed bound on $F_+$. Thus this case does not occur. Let $(i_j, q_j)$ be the $s$ pairs guaranteed by Lemma~\ref{lem-many-dense-stars}. Let $\mathcal B := \{x_{i_j} \neq q_{j} \, \forall j\}$. We can without loss of generality assume that $\mathcal B$ is as in Lemma~\ref{lem:outdeg} with $a = s$, and we shall use notation from that lemma. Note that since $F$ is not isomorphic to $[0,s-1]\times\Z_k^{n-1}$, $\beta > 0$. Also we have that
  \begin{multline*}
  |F''| \geq \sum_{j = 1}^s|F \cap \{x_{i_j} = q_{j}\}| - \sum_{t < j}|F \cap \{x_{i_j} = q_{j}\}\cap \{x_{i_t} = q_{t}\}| \geq\\\geq sk^{n-1}(\frac{s}{s+1} - 4\sqrt{\frac{s}{k-1}}) - \frac{s^2}{2}k^{n-2} \geq sk^{n-1}(\frac{s}{s+1} - 5\sqrt{\frac{s}{k-1}}).
  \end{multline*}
  Thus, by Lemma~\ref{lem:outdeg},
  \begin{multline*}
  e(F'', F') \geq |F'|(k-1)^{n-1}(1 - \frac{s}{k-1})\Big(k\alpha'' - \frac{\sqrt{sk}}{k-s-1}\Big) \geq \\\geq |F'|(k-1)^{n-1}(1 - \frac{s}{k-1})\Big(\frac{s^2}{s+1} - 5\sqrt{\frac{s^3}{k-1}} - \frac{\sqrt{sk}}{k-s-1}\Big) \geq\\\geq |F'|(k-1)^{n-1}\Bigg(\frac{s^2}{s+1} - 7\sqrt{\frac{s^3}{k-1}}\Bigg), 
  \end{multline*}
  and so $$\Delta(K_{k}^{\otimes n}[F]) \geq (k-1)^{n-1}\Bigg(\frac{s^2}{s+1} - 7\sqrt{\frac{s^3}{k-1}}\Bigg).$$
\end{proof}

\begin{Claim}\label{cl:easy-prob-constr-k}
  Suppose that $\Delta_{K_k^{\otimes n}}(m)=d$ with $d\geq2$. Then, for every $p\in(0,1)$,
  $$\Delta_{K_k^{\otimes n}}\!\left(\left\lceil pm\left(1-\frac{1}{d^2}\right)\right\rceil\right) \leq pd + \sqrt{d\ln d}.$$
\end{Claim}
\begin{proof}
  Let $G$ denote the ambient product graph, and choose $F\subseteq V(G)$ with $|F|=m$ and $\Delta(G[F])=d$. Let $F_p$ be a $p$-random subset of $F$, and set
  $$B:=\{x\in F_p:\deg_{F_p}(x)>pd+\sqrt{d\ln d}\}.$$
  Conditional on $x\in F_p$, the random variable $\deg_{F_p}(x)$ is a sum of $\deg_F(x)$ independent Bernoulli variables. Hence, for $t>0$ and $\deg_F(x)>0$, the Chernoff--Hoeffding bound gives
  $$\Pp[x\in B]\leq p\,\Pp[\deg_{F_p}(x)\geq p\deg_F(x)+t\mid x\in F_p]
  \leq p\exp\!\left(-\frac{2t^2}{\deg_F(x)}\right)
  \leq p\exp\!\left(-\frac{2t^2}{d}\right).$$
  (The assertion is trivial when $\deg_F(x)=0$.) Taking $t=\sqrt{d\ln d}$ and summing over $x\in F$ yields $\mathbb{E}|B|\leq mp/d^2$. Therefore
  $$\mathbb{E}(|F_p|-|B|)\geq pm\left(1-\frac{1}{d^2}\right).$$
  For some realization, deleting $B$ leaves at least $\left\lceil pm(1-d^{-2})\right\rceil$ vertices and maximum degree at most $pd+\sqrt{d\ln d}$. Taking a subset of the required size proves the claim.
\end{proof}

\begin{Lemma}
  For any $s < \lambda \leq s + 1$ we have $$\Delta_{K_k^{\otimes n}}(\lambda\cdot k^{n-1}) \geq (k-1)^{n-1}\frac{\lambda s}{s+1}\Bigg(1 - 20\sqrt\frac{\ln ((k-1)^{n-1})}{(k-1)^{n-1}}\Bigg)\Bigg(1 - 7\sqrt{\frac{(s+1)^2}{s(k-1)}}\Bigg).$$
\end{Lemma}

\begin{proof}
  Let $D:=(k-1)^{n-1}$ and suppose, for a contradiction, that some $F$ of size $\lambda k^{n-1}$ induces maximum degree $d$ strictly smaller than the asserted lower bound. If $20\sqrt{\ln D/D}\geq1$, the asserted bound is nonpositive and there is nothing to prove, so assume otherwise. Lemma~\ref{lem-jumps} and the standing assumption $k\geq1000s^2$ give $d\geq D/4$.

Apply Claim~\ref{cl:easy-prob-constr-k} with
$$p=\left(1+\frac{2}{d^2}\right)\frac{s}{\lambda}.$$
As before, if $p\geq1$, the conclusion follows directly from Lemma~\ref{lem-jumps} and monotonicity, so we may assume $p<1$. Then $p|F|(1-d^{-2})>sk^{n-1}$. Moreover, using $d\geq D/4$, $s/\lambda\geq1/2$, and $d\leq(k-1)^n$, we obtain
$$
\left(1+\frac{2}{d^2}\right)\frac{s}{\lambda}
+\sqrt{\frac{\ln d}{d}}
\leq\frac{s}{\lambda}\left(1+10\sqrt{\frac{\ln D}{D}}\right).
$$
Therefore
$$
pd+\sqrt{d\ln d}
< D\frac{s^2}{s+1}
\left(1-7\sqrt{\frac{(s+1)^2}{s(k-1)}}\right)
=D\left(\frac{s^2}{s+1}-7\sqrt{\frac{s^3}{k-1}}\right),
$$
contradicting Lemma~\ref{lem-jumps}.
\end{proof}

\begin{Lemma}
  For any $s < \lambda \leq s + 1$ we have
  $$\Delta_{K_k^{\otimes n}}(\lambda\cdot k^{n-1}) \leq (k-1)^{n-1}\frac{\lambda s}{s+1}\Bigg(1 + 2\sqrt\frac{\ln ((k-1)^{n-1})}{(k-1)^{n-1}}\Bigg).$$
\end{Lemma}

\begin{proof}
  Let $m=(s+1)k^{n-1}$ and $d=s(k-1)^{n-1}$. If
  $$p:=\frac{\lambda}{s+1}\left(1+\frac{2}{d^2}\right)<1,$$
  apply Claim~\ref{cl:easy-prob-constr-k} with these values. If $p\geq1$, monotonicity and the exact construction at size $m$ give $\Delta_{K_k^{\otimes n}}(\lambda k^{n-1})\leq d$. Moreover, $p\geq1$ and $\sqrt{\ln d/d}\geq d^{-2}$ imply
  $$\frac{\lambda}{s+1}\left(1+2\sqrt{\frac{\ln d}{d}}\right)\geq1,$$
  so $d$ is no larger than the displayed upper bound.
\end{proof}

\subsection{The Hamming product}

\paragraph{A specific limit.}

\begin{Claim}
For fixed $s\leq\lambda\leq s+1$,
  $$\lim_{k \to \infty}\lim_{n\to \infty}\frac{\Delta_{K_{k}^{\Box n}}(\lambda\cdot k^{n-1})}{n} = s\Big(2 - \frac{s+1}{\lambda}\Big).$$
\end{Claim}

\begin{proof}
  Let $q = \lceil \frac{k}{\lambda}\rceil - 1$. By Statement~\ref{st}, $\lim_{\alpha \to 
  \frac{1}{q}-0}\lim_{n\to \infty}\frac{\Delta_{K_k^{\Box n}}(\alpha k^{n})}{(k-1)n} = \bar a_k(\frac{1}{q})$. Since $\frac{\lambda}{k} < \frac{1}{q}$, we have $\lim_{n\to \infty}\frac{\Delta_{K_k^{\Box n}}(\lambda k^{n-1})}{(k-1)n} \leq \bar a_k(\frac{1}{q})$. It is easy to see that the derivative $\bar a_k'(\alpha) \leq 3$ on each interval. Consequently, $\bar a_k(\frac{1}{q}) \leq \bar a_k(\frac{\lambda}{k}) + 3(\frac1q - \frac{\lambda}{k}) \leq \bar a_k(\frac{\lambda}{k}) + \frac{6\lambda^2}{k^2}$ for sufficiently large $k$. Then $\lim_{k \to \infty}\lim_{n\to \infty}\frac{\Delta_{K_{k}^{\Box n}}(\lambda\cdot k^{n-1})}{n} \leq \lim_{k \to \infty}(k-1)\bar a_k(\frac{\lambda}{k}) = s\Big(2 - \frac{s+1}{\lambda}\Big)$. The converse inequality follows from the average degree bound.
\end{proof}

\paragraph{Big subsets of maximum degree 1 for $k$ with small factors.}

The following theorem was proved in~\cite{HammingMatching}.

\begin{Theorem}
  For any $k \geq 3$ and any $n$ there exists an $F_k^n \subseteq \Z_k^n$ with $|F_k^n| = k^{n-1} + 1$ and $\Delta(K_k^{\Box n}[F_k^n]) = 1$.
\end{Theorem}
All analogous constructions found in the literature have size $k^{n-1}+1$. This can be improved whenever $k$ has a factor other than~$2$.

\begin{Claim}
If $p\mid k$ for some $p\geq3$, then there exists an induced subgraph of $K_k^{\Box n}$ of maximum degree 1 and size $\geq k^{n-1} + (\frac{k}{p})^{n-1}$.
\end{Claim}

\begin{proof}
  Set $t:=k/p$. Fix a bijection between $[0,k-1]$ and $[0,p-1]\times[0,t-1]$, and use it to write each element of $[0,k-1]$ as a pair. Let $F_p^n$ be the set supplied by the preceding theorem, and let $A_t^n:=\{x\in[0,t-1]^n:\sum_i x_i\equiv0\pmod t\}$. Define

  $$F = \{((x_1,y_1),(x_2,y_2),\ldots,(x_n,y_n)) : (x_1,\ldots,x_n) \in F_p^n, (y_1,\ldots,y_n)\in A^n_t\}.$$

  Clearly, the subgraph induced by $F$ has maximum degree at most~$1$ and $|F| = (p^{n-1} + 1)t^{n-1} = k^{n-1} + (\frac{k}{p})^{n-1}$.
\end{proof}

\end{document}